%% file: main.tex
\documentclass[10pt]{article}
\usepackage{amsmath}
\usepackage{amsfonts}
\usepackage{amssymb}
\usepackage{amsthm}
\usepackage{enumerate}
\usepackage[totalwidth=450pt, totalheight=650pt]{geometry}

\begin{document}

\newcommand{\C}{{\mathbb{C}}}
\newcommand{\R}{{\mathbb{R}}}
\newcommand{\Q}{B}
\newcommand{\Z}{{\mathbb{Z}}}
\newcommand{\N}{{\mathbb{N}}}
\newcommand{\q}{\left}
\newcommand{\w}{\right}
\newcommand{\Ninf}{\N_\infty}
\newcommand{\Vol}[1]{\mathrm{Vol}\q(#1\w)}
\newcommand{\B}[4]{B_{\q(#1,#2\w)}\q(#3,#4\w)}
\newcommand{\CjN}[3]{\q\|#1\w\|_{C^{#2}\q(#3\w)}}
\newcommand{\Cj}[2]{C^{#1}\q( #2\w)}
\newcommand{\grad}{\bigtriangledown}
\newcommand{\sI}[2]{\mathcal{I}\q(#1,#2 \w)}
\newcommand{\Det}[1]{\det_{#1\times #1}}
\newcommand{\sK}{\mathcal{K}}
\newcommand{\sKt}{\widetilde{\mathcal{K}}}
\newcommand{\sA}{\mathcal{A}}
\newcommand{\sB}{\mathcal{B}}
\newcommand{\sC}{\mathcal{C}}
\newcommand{\sD}{\mathcal{D}}
\newcommand{\sS}{\mathcal{S}}
\newcommand{\sL}{\mathcal{L}}
\newcommand{\sF}{\mathcal{F}}
\newcommand{\sQ}{\mathcal{Q}}
\newcommand{\sV}{\mathcal{V}}
\newcommand{\sM}{\mathcal{M}}
\newcommand{\sT}{\mathcal{T}}
\newcommand{\sPh}{\widehat{\sP}}
\newcommand{\sNh}{\widehat{\sN}}
\newcommand{\cV}{\q( \sV\w)}
\newcommand{\vsig}{\varsigma}
\newcommand{\vsigt}{\widetilde{\vsig}}
\newcommand{\dil}[2]{#1^{\q(#2\w)}}
\newcommand{\dilp}[2]{#1^{\q(#2\w)_{p}}}
\newcommand{\lA}{-\log_2 \sA}
\newcommand{\eh}{\widehat{e}}
\newcommand{\Ho}{\mathbb{H}^1}
\newcommand{\sd}{\sum d}
\newcommand{\dt}{\tilde{d}}
\newcommand{\dhc}{\hat{d}}
\newcommand{\Span}[1]{\mathrm{span}\q\{ #1 \w\}}
\newcommand{\dspan}[1]{\dim \Span{#1}}
\newcommand{\K}{K_0}
\newcommand{\ad}[1]{\mathrm{ad}\q( #1 \w)}
\newcommand{\LtOpN}[1]{\q\|#1\w\|_{L^2\rightarrow L^2}}
\newcommand{\LpOpN}[2]{\q\|#2\w\|_{L^{#1}\rightarrow L^{#1}}}
\newcommand{\LpN}[2]{\q\|#2\w\|_{L^{#1}}}
\newcommand{\Jac}{\mathrm{Jac}\:}
\newcommand{\kapt}{\widetilde{\kappa}}
\newcommand{\gt}{\widetilde{\gamma}}
\newcommand{\gtt}{\widetilde{\widetilde{\gamma}}}
\newcommand{\gh}{\widehat{\gamma}}
\newcommand{\Sh}{\widehat{S}}
\newcommand{\Wh}{\widehat{W}}
\newcommand{\sMh}{\widehat{\sM}}
\newcommand{\Ih}{\widehat{I}}
\newcommand{\Wt}{\widetilde{W}}
\newcommand{\Xt}{\widetilde{X}}
\newcommand{\Tt}{\widetilde{T}}
\newcommand{\Dt}{\widetilde{D}}
\newcommand{\Phit}{\widetilde{\Phi}}
\newcommand{\Vh}{\widehat{V}}
\newcommand{\Xh}{\widehat{X}}
\newcommand{\ch}{\widehat{c}}
\newcommand{\deltah}{\widehat{\delta}}
\newcommand{\sSh}{\widehat{\mathcal{S}}}
\newcommand{\sFh}{\widehat{\mathcal{F}}}
\newcommand{\thetah}{\widehat{\theta}}
\newcommand{\ct}{\widetilde{c}}
\newcommand{\at}{\tilde{a}}
\newcommand{\bt}{\tilde{b}}
\newcommand{\fg}{\mathfrak{g}}
\newcommand{\cC}{\q( \sC\w)}
\newcommand{\cG}{\q( \sC_{\fg}\w)}
\newcommand{\cJ}{\q( \sC_{J}\w)}
\newcommand{\cY}{\q( \sC_{Y}\w)}
\newcommand{\cYu}{\q( \sC_{Y}\w)_u}
\newcommand{\cJu}{\q( \sC_{J}\w)_u}
\newcommand{\Bb}{\overline{B}}
\newcommand{\Qb}{\overline{Q}}
\newcommand{\sP}{\mathcal{P}}
\newcommand{\sN}{\mathcal{N}}
\newcommand{\cH}{\q(\mathcal{H}\w)}
\newcommand{\Omegat}{\widetilde{\Omega}}
\newcommand{\Kt}{\widetilde{K}}
\newcommand{\sMt}{\widetilde{\sM}}
\newcommand{\denum}[2]{#1_{#2}}
\newcommand{\phit}{\tilde{\phi}}
\newcommand{\nuset}{\q\{1,\ldots, \nu\w\}}
\newcommand{\diam}[1]{ {\mathrm{diam}}\q\{#1\w\} }
\newcommand{\Nt}{\widetilde{N}}
\newcommand{\psit}{\widetilde{\psi}}
\newcommand{\sigt}{\widetilde{\sigma}}
\newcommand{\muoS}{\q\{\mu_1\w\}}
\newcommand{\Mjcutoff}{\nu+2}
\newcommand{\LplqOpN}[3]{\q\| #3 \w\|_{L^{#1}\q(\ell^{#2}\q(\N^{\nu} \w) \w)\rightarrow L^{#1}\q( \ell^{#2 }\q( \N^{\nu} \w) \w) }}
\newcommand{\LplqN}[3]{\q\| #3 \w\|_{L^{#1}\q(\ell^{#2 }\q(\N^{\nu} \w) \w) }}
\newcommand{\iinf}{\iota_\infty}
\newcommand{\jz}{\ell}
\newcommand{\Lpp}[2]{L^{#1}\q(#2\w)}
\newcommand{\Lppn}[3]{\q\|#3\w\|_{\Lpp{#1}{#2}}}
\newcommand{\Ewmu}{E\cup \q\{\mu_1\w\}}
\newcommand{\Ewmuc}{\q(\Ewmu\w)^c}
\newcommand{\ip}[2]{\q< #1, #2 \w>}
\newcommand{\hd}{\widehat{d}}
\newcommand{\mset}{\q\{ 1,\ldots, m\w\}}
\newcommand{\expL}[1]{\exp_L\q( #1\w)}

\newtheorem{thm}{Theorem}[section]
\newtheorem{cor}[thm]{Corollary}
\newtheorem{prop}[thm]{Proposition}
\newtheorem{lemma}[thm]{Lemma}
\newtheorem{conj}[thm]{Conjecture}

\theoremstyle{remark}
\newtheorem{rmk}[thm]{Remark}

\theoremstyle{definition}
\newtheorem{defn}[thm]{Definition}

\theoremstyle{remark}
\newtheorem{example}[thm]{Example}

\numberwithin{equation}{section}

\title{Multi-parameter singular Radon transforms III: real analytic surfaces}
\author{Elias M. Stein\footnote{Partially supported by NSF DMS-0901040.} and Brian Street\footnote{Partially supported by NSF DMS-0802587.}}
\date{}

\maketitle

\begin{abstract}
\input{abstract}
\end{abstract}

\section{Introduction}
\input{intro}

\section{A motivating special case}
\input{specialcase}

\section{Kernels}\label{SectionKernels}
\input{kernels}

\section{Multi-parameter Carnot-Carath\'eodory geometry}\label{SectionCC}
\input{cc}

\section{Results}
\input{results}

\section{Past work}
\input{pastwork}

\section{A more general maximal function}\label{SectionMoreGenMax}
\input{moregenmax}

\section{When $\gamma$ is $C^\infty$: the results of \cite{SteinStreetMultiParameterSingRadonLp}}\label{SectionSSLp}
\input{lpres}

\section{A primer on real analytic functions}\label{SectionRealAnal}
\input{realanal}

\section{Reduction to the $C^\infty$ case}\label{SectionReduction}
\input{redcinf}

	\subsection{Singular Radon Transforms}\label{SectionReductionSingInt}
	\input{redsingint2}


    \subsection{Maximal Radon
    transforms}\label{SectionReductionMaxInt}
    \input{redmax}

\section{Proof of the general maximal result (Theorem
\ref{ThmMoreGenMaxFunc})}\label{SectionProofOfMax}
\input{proofofmax}

\section{A closing remark}\label{SectionClosing}
\input{closing}


\bibliographystyle{amsalpha}

\bibliography{radon}



\center{MCS2010: Primary 42B20, Secondary 42B25, 26E05, 32B05}

\center{Keywords: Calder\'on-Zygmund theory, singular integrals, singular
Radon transforms, 
real analytic surfaces, Weierstrass preparation,
maximal Radon transforms, Littlewood-Paley theory, product
kernels, flag kernels, Carnot-Carath\'eodory geometry}

\end{document}

%% file: abstract.tex
The goal of this paper is to study operators of the form,
\[
Tf(x)= \psi(x)\int f(\gamma_t(x))K(t)\: dt,
\]
where $\gamma$ is a real analytic function defined on a neighborhood
of the origin in $(t,x)\in \R^N\times \R^n$,
satisfying $\gamma_0(x)\equiv x$, $\psi$ is a cutoff
function supported near $0\in \R^n$, and $K$ is a ``multi-parameter
singular kernel'' supported near $0\in \R^N$.
A main example is when $K$ is a ``product kernel.''
We also study maximal operators of the form,
\[
\mathcal{M} f(x) = \psi(x)\sup_{0<\delta_1,\ldots, \delta_N<<1} \int_{|t|<1} |f(\gamma_{\delta_1 t_1,\ldots,\delta_N t_N}(x))|\: dt.
\]
We show that $\mathcal{M}$ is bounded on $L^p$ ($1<p\leq \infty$).
We give conditions on $\gamma$ under which $T$ is bounded on $L^p$ ($1<p<\infty$); these conditions hold automatically when $K$ is a Calder\'on-Zygmund kernel.  This is the final paper in a three part series.  The first two
papers consider the more general case when $\gamma$ is $C^\infty$.

%% file: intro.tex
In this paper we consider operators of the form,
\begin{equation}\label{EqnIntroOpForm}
Tf\q( x\w) = \psi\q( x\w) \int f\q( \gamma_t\q( x\w)\w) K\q( t\w)\: dt,
\end{equation}
where $\psi\in C_0^\infty\q( \R^n\w)$ is supported near $0$,
$\gamma_t\q( x\w):\R_0^N\times \R_0^n\rightarrow \R^n$
is a germ of a real analytic function (defined on a neighborhood of
$\q(0,0\w)$) satisfying $\gamma_0\q( x\w)\equiv x$,\footnote{Here we
write $f:\R_0^N\rightarrow \R^m$ to denote that $f$ is a germ
of a function defined on a neighborhood of $0$.}
and $K$ is a ``multi-parameter'' distribution kernel, supported
near $0\in \R^N$.  For instance, one could take $K$ to be
a ``product kernel'' supported near $0$.\footnote{Our main theorem
applies to kernels more general than product kernels.}
To define this notion, suppose we have decomposed $\R^N=\R^{N_1}\times\cdots\times \R^{N_{\nu}}$.
A product kernel satisfies
\begin{equation}\label{EqnProductKernelCond}
\q|\partial_{t_1}^{\alpha_1}\cdots \partial_{t_\nu}^{\alpha_\nu} K\q( t\w)\w|\lesssim \q|t_1\w|^{-N_1-\q|\alpha_1\w|}\cdots \q|t_{\nu}\w|^{-N_{\nu}-\q|\alpha_\nu\w|},
\end{equation}
along with certain ``cancellation conditions'' (see Section 16 of
\cite{StreetMultiParameterSingRadonLt}).\footnote{The simplest
example of a product kernel is given by
$K\q(t_1,\ldots, t_\nu\w)=K_1\q(t_1\w)\otimes\cdots\otimes K_\nu\q(t_\nu\w)$,
where $K_1,\ldots, K_\nu$ are Calder\'on-Zygmund kernels.  That is, they
satisfy $\q|\partial_{t_j}^{\alpha} K_j\q(t_j\w)\w|\lesssim \q|t_j\w|^{-N_j-\q|\alpha\w|}$, again along with certain ``cancellation conditions.''  When
$\nu=1$, the class of product kernels is exactly the class of
Calder\'on-Zygmund kernels.  For a precise statement
of these cancellation conditions, see Section 16 of \cite{StreetMultiParameterSingRadonLt}; we do not make it precise in this paper, since we deal
with more general kernels.}
We will also study maximal operators of the form,
\begin{equation*}
\sM f\q( x\w) = \sup_{0<\delta_1,\ldots,\delta_{\nu}\leq a} \psi\q( x\w) \int_{\q|t\w|\leq 1} \q|f\q( \gamma_{\q(\delta_1 t_1,\ldots, \delta_\nu t_\nu\w)}\q( x\w)\w)\w|\: dt_1\: \cdots \: dt_{\nu},
\end{equation*}
where $\psi$ is as before with $\psi\geq 0$, and $a>0$ is assumed to be small.

First we describe our results in the single parameter case ($\nu=1$).
In that case, we consider $K$ to be a standard Calder\'on-Zygmund
kernel supported near $0$ (a Calder\'on-Zygmund kernel is the special
case of product kernels with $\nu=1$).
In this case, the operator in \eqref{EqnIntroOpForm} is bounded on $L^p$ ($1<p<\infty$)
with no additional assumptions (provided $\psi$ and $K$ have
sufficiently small support, depending on $\gamma$).
Similarly, when $\nu=1$, $\sM$ is bounded on $L^p$ ($1<p\leq \infty$) (provided
$\psi$ has small enough support and $a>0$ is sufficiently small).

When we move to the multi-parameter case, the study of $T$ and
$\sM$ diverge.
The results for $\sM$ are simple to state:  just as in the single-parameter
case, $\sM$ is bounded on $L^p$ ($1<p\leq \infty$) with no additional
assumptions.  In fact, this will follow from proving the
$L^p$ boundedness for even stronger maximal operators.

Unfortunately, the results for $T$ are not so simple.
Indeed, for $\gamma:\R^2\times \R\rightarrow \R$,
given by $\gamma_{\q( s,t\w)}\q( x\w)= x-st$,
there are product kernels $K\q(s,t\w)$ of arbitrarily
small support such that $T$ is not bounded on $L^2$
(this was first noted in \cite{NagelWaingerL2BoundednessOfHilbertTransformsMultiParameterGroup}, see also Section 17.5 of \cite{StreetMultiParameterSingRadonLt}).
Thus, it is necessary to introduce additional assumptions on $\gamma$
to obtain the $L^p$ boundedness of $T$.

We now describe a special case of our results, for $\nu$-parameter
product kernels $K\q( t_1,\ldots, t_{\nu}\w)$ (thus satisfying
\eqref{EqnProductKernelCond}).
In
\cite{ChristNagelSteinWaingerSingularAndMaximalRadonTransforms}
it was shown that $\gamma$ could be written asymptoticly
in the form,\footnote{\eqref{EqnIntroAsym} simply means
$\gamma_t\q( x\w) = \exp\q(\sum_{0<\q|\alpha\w|< M}t^{\alpha} X_\alpha \w)x+O\q( \q|t\w|^M\w)$, for every $M$ as $t\rightarrow 0$.  In particular, the reader may wish to consider just the case
when $\gamma_t\q( x\w)=\exp\q(\sum_{0<\q|\alpha\w|\leq L} t^{\alpha}X_\alpha\w)x$,
and the $X_\alpha$ are germs of real analytic vector fields.}
\begin{equation}\label{EqnIntroAsym}
\gamma_t\q(x\w) \sim \exp\q(\sum_{\q|\alpha\w|>0} t^{\alpha}X_\alpha \w)x,
\end{equation}
where each $X_\alpha$ is a real analytic vector field.
Separate each multi-index $\alpha=\q( \alpha_1,\ldots, \alpha_\nu\w)$,
where $\alpha_\mu\in \N^{N_{\mu}}$ is a multi-index, and
$t^{\alpha}=t_1^{\alpha_1}\cdots t_{\nu}^{\alpha_\nu}$.
We call $\alpha$ a {\it pure power} if $\alpha_\mu\ne 0$ for only one $\mu$.
Otherwise, we call $\alpha$ a {\it non-pure power}.

A special case of our theorem is as follows:  if $X_\alpha=0$ for
every non-pure power $\alpha$, then $T$ is bounded on $L^p$ ($1<p<\infty$).  In
the single parameter case, every power is a pure power, and so this
subsumes the single-parameter result.
In fact, we will be able to deal with some cases when the non-pure
powers are not necessarily zero.  Our assumption will be that
the pure powers ``control'' the non-pure powers, in an appropriate sense.

This paper is the third in a series. The first two
\cite{StreetMultiParameterSingRadonLt,SteinStreetMultiParameterSingRadonLp}
dealt with the more general situation when $\gamma$ is $C^\infty$,
instead of real analytic. The theorems in those papers took a rather
complicated form. We will see that after an appropriate
``preparation theorem,\footnote{The preparation theorem we need is
a Weierstrass type preparation theorem due to Galligo
\cite{GalligoTheoremeDeDivisionEtStabiliteEnGeometrieAnalytiqueLocale}.}''
the main result in this paper for the singular Radon transform $T$
is actually a special case of the results in
\cite{SteinStreetMultiParameterSingRadonLp}. The idea is that when
$\gamma$ is assumed to be real analytic, many of the assumptions in
\cite{SteinStreetMultiParameterSingRadonLp} come for free.
See \cite{SteinStreetMultiparameterSingularRadonTransformsAnnounce}
for an announcement of this series, and an overview tying all
three papers together.

The maximal operator $\sM$ is not a special case of the results in
\cite{SteinStreetMultiParameterSingRadonLp}. Indeed, we will prove a
new maximal result concerning $C^{\infty}$ $\gamma$ (see Section
\ref{SectionMoreGenMax}) which will imply the maximal result for
$\sM$.  While this result was not covered in
\cite{SteinStreetMultiParameterSingRadonLp}, we will see that many
of the methods can be transfered to this situation, and the main
outline of the proof is quite similar.

%% file: specialcase.tex
In this section we explain our argument in a special motivating case,
which contains an essential point which we will use (in various forms) throughout
the paper.

First we describe a special case
of the results in
\cite{ChristNagelSteinWaingerSingularAndMaximalRadonTransforms}
in the $C^\infty$ context.
Indeed, suppose for each multi-index $\alpha$, $0< \q|\alpha\w|\leq L$, we are given
a $C^\infty$ vector fields $X_\alpha$ defined on a neighborhood of $0$.
Suppose that this collection of vector fields satisfies H\"omander's
condition at $0$:  the set of $X_\alpha$ along with all their
commutators of all orders spans the tangent space at $0$.
Define a function $\gamma$ by,
\begin{equation*}
\gamma_t\q( x\w) = \exp\q(\sum_{0<\q|\alpha\w|\leq L} t^{\alpha} X_\alpha \w)x.
\end{equation*}
It is a theorem of
Christ, Nagel, Stein, and Wainger \cite{ChristNagelSteinWaingerSingularAndMaximalRadonTransforms}
that the operator,
\begin{equation}\label{EqnSpecialCaseOp}
f\mapsto \psi\q( x\w) \int f\q( \gamma_t\q( x\w)\w) K\q( t\w)\: dt
\end{equation}
is bounded on $L^p$ ($1<p<\infty$), for every standard
Calder\'on-Zygmund kernel supported on a sufficiently small
neighborhood of $0$, and for $\psi\in C_0^\infty$, supported on a
sufficiently small neighborhood of $0$.

It was discussed in Section 3 of
\cite{StreetMultiParameterSingRadonLt} that one need not assume the
$X_\alpha$ satisfy H\"ormander's condition.  Instead, one may assume
the weaker condition that the involutive distribution generated by
the $X_\alpha$ is locally finitely
generated as a $C^\infty$ module.  For
then, the standard Frobenius theorem holds and foliates the ambient
space into leaves; the $X_\alpha$ satisfying H\"ormander's condition
on each leaf. As was discussed in Section 3 of
\cite{StreetMultiParameterSingRadonLt}, the methods of
\cite{ChristNagelSteinWaingerSingularAndMaximalRadonTransforms} are
not sufficient to obtain the $L^p$ boundedness of
\eqref{EqnSpecialCaseOp} in this case. Nevertheless, the $L^p$
($1<p<\infty$) boundedness holds and is a special case of the
results in \cite{SteinStreetMultiParameterSingRadonLp}.

Now we specialize to the case when the vector fields $X_\alpha$ are
real analytic.  The involutive distribution generated by a finite collection of real analytic vector fields
is {\it always} locally finitely generated
as a $C^\infty$ module.
  This fact seems to have first been noted in
\cite{NaganoLinearDifferentialSystemsWithSingularities,LobryControlabiliteDesSystemesNonLinearies},
see Section \ref{SectionRealAnal} for a further discussion. Thus,
when the vector fields are real analytic, the $L^p$ boundedness of
\eqref{EqnSpecialCaseOp} holds. This idea is the core of this entire
paper, and similar arguments will be used throughout.

%% file: kernels.tex
In this section, we will discuss the classes of kernels
$K\q( t\w)$ for which we will study operators
of the form \eqref{EqnIntroOpForm}.
The kernels which we study will be supported in
$\Q^N\q( a\w)$, where $a>0$ is some small number
to be chosen later (depending on $\gamma$).
Fix $\nu\in \N$, we will be studying $\nu$ parameter
operators.  

We suppose we are given $\nu$-parameter dilations on $\R^N$.  That is, 
we are
given $e=\q( e_1,\ldots, e_N\w)$, with each $0\ne e_j= \q( e_j^1,\ldots, e_j^{\nu}\w)\in \N^\nu$ (here, $0\in \N$).
For $\delta\in \q[0,\infty\w)^\nu$ and $t=\q( t_1,\ldots, t_N\w)\in \R^N$,
we define,\footnote{Here $\delta^{e_j}$ is defined via standard multi-index
notation: $\delta^{e_j}=\prod_\mu \delta_\mu^{e_j^\mu}$.}
\begin{equation}\label{EqnDefndeltat}
\delta t = \q( \delta^{e_1}t_1,\ldots, \delta^{e_{N}}t_N \w),
\end{equation}
thereby obtaining $\nu$-parameter dilations on $\R^N$.
For each $\mu$, $1\leq \mu\leq \nu$, let $t_\mu$ denote those coordinates
$t_j$
of $t=\q(t_1,\ldots, t_N \w)\in \R^N$ such that $e_j^\mu\ne 0$.

The class of distributions we will define depends on 
$N$, $a$, $e$, and $\nu$. 
Given a function $\vsig$ on $\R^N$, and $j\in \N^\nu$, define,
\begin{equation*}
\dil{\vsig}{2^j}\q( t\w) = 2^{j\cdot e_1+\cdots + j\cdot e_N} \vsig\q( 2^j t\w).
\end{equation*}
Note that $\dil{\vsig}{2^j}$ is defined in such a way that,
\begin{equation*}
\int \dil{\vsig}{2^j} \q( t\w) \: dt = \int \vsig\q( t\w) \: dt.
\end{equation*}

\begin{defn}\label{DefnsK}
We define $\sK=\sK\q( N,e,a,\nu\w)$ to be the set of all distributions,
$K$, of the form
\begin{equation}\label{EqnSumDefK}
K=\sum_{j\in \N^{\nu}} \dil{\vsig_j}{2^j},
\end{equation}
where $\q\{\vsig_j\w\}_{j\in\N^\nu}\subset C_0^\infty\q( \Q^N\q( a\w) \w)$ is a bounded
set, satisfying
\begin{equation*}
\int \vsig_j\q( t\w) \: dt_\mu =0,\quad 0\ne j_\mu.
\end{equation*}
It was shown in \cite{StreetMultiParameterSingRadonLt} that
any sum of the form \eqref{EqnSumDefK} converges in the
sense of distributions.
\end{defn}

See \cite{StreetMultiParameterSingRadonLt} for a more in-depth
discussion of the class $\sK$.

\begin{rmk}
The class of kernels studied in
\cite{SteinStreetMultiParameterSingRadonLp}
was slightly more general:  it was allowed to depend
on another parameter $\mu_0$, $1\leq \mu_0\leq \nu$,
and the coordinates of each $e_j$ could be
elements of $\q[0,\infty\w)$, instead of $\N$.
The results in this paper can be extended to deal with
that case as well (with essentially no additional work),
but we state the results in this simpler case for
clarity.
See Section \ref{SectionClosing} for some comments on this.
\end{rmk}

%% file: cc.tex
To state our theorem regarding the singular Radon transforms given
by \eqref{EqnIntroOpForm} in full generality, we must introduce the
notion of Carnot-Carath\'eodory geometry; this notion played an
essential role in
\cite{StreetMultiParameterSingRadonLt,SteinStreetMultiParameterSingRadonLp}.
Our main reference for Carnot-Carath\'eodory geometry is
\cite{StreetMultiParameterCCBalls}, and we refer the reader there
for more information.  In this section we introduce only the most
basic definitions associated with Carnot-Carath\'eodory geometry.

Let $\Omega\subseteq \R^n$ be an open set, and let $X_1,\ldots, X_q$
be $C^\infty$ vector fields on $\Omega$.  Denote this list of vector
fields by $X$. We define the Carnot-Carath\'eodory ball, centered at
$x_0\in \Omega$, of unit radius, with respect to the list of vector
fields $X$ by
\begin{equation*}
\begin{split}
B_X\q( x_0\w):=\bigg\{y\in \Omega \:\bigg|\: &\exists \gamma:\q[0,1\w]\rightarrow \Omega, \gamma\q( 0\w) =x_0, \gamma\q( 1\w) =y,  \\
&\gamma'\q( t\w) = \sum_{j=1}^q a_j\q( t\w) X_j\q( \gamma\q(
t\w)\w),
 a_j\in L^\infty\q( \q[0,1\w]\w), \\
&\Lppn{\infty}{\q[0,1\w]}{\q(\sum_{1\leq j\leq q}
\q|a_j\w|^2\w)^{\frac{1}{2}}}<1\bigg\}.
\end{split}
\end{equation*}
Now that we have the definition of balls with unit radius, we may
define (multi-parameter) balls of any radius merely by scaling the
vector fields. To do so, we assign to each vector field, $X_j$, a
(multi-parameter) formal degree $0\ne d_j=\q(d_j^1,\ldots, d_j^\nu
\w)\in\N^\nu$.  For $\delta=\q(\delta_1,\ldots, \delta_\nu \w)\in
\q[0,\infty\w)^\nu$, we define the list of vector fields $\delta X$
to be the list $\q( \delta^{d_1} X_1,\ldots, \delta^{d_q}X_q\w)$.
Here, $\delta^{d_j}$ is defined by the standard multi-index
notation:  $\delta^{d_j} =\prod_{\mu=1}^\nu \delta_\mu^{d_j^\mu}$.
We define the ball of radius $\delta$ centered at $x_0\in \Omega$ by
$$\B{X}{d}{x_0}{\delta} := B_{\delta X}\q( x_0\w).$$

\begin{defn}\label{DefnFiniteSetControl}
Let $\q(X,d\w) = \q( X_1,d_1\w),\ldots, \q( X_q,d_q\w)$ be a finite
list of $C^\infty$ vector fields with multi-parameter formal degrees
as above. Fix $x_0\in \Omega$. Let $\q( X_0,d_0\w)$ be another
$C^\infty$ vector field with multi-parameter formal degree $0\ne
d_0\in \N^{\nu}$.  We say that $\q( X,d\w)$ {\it controls} $\q(
X_0,d_0\w)$ on a neighborhood of $x_0$ if there exists an open set
$U$ with $x_0\in U\subseteq \Omega$, and $\tau_1>0$ such that for
every $\delta \in \q[0,1\w]^\nu$, $x\in U$, there exist
$c_{x,j}^\delta \in C^0\q(\B{X}{d}{x}{\tau_1\delta} \w)$ ($1\leq j \leq
q$) such that,
\begin{itemize}
\item $\delta^{d_0} X_0 = \sum_{j=1}^q c_{x,j}^{\delta}
\delta^{d_j} X_j,$ on $\B{X}{d}{x}{\tau_1\delta}$.
\item $\sup_{\substack{\delta\in \q[0,1\w]^\nu\\ x\in U}}\sum_{\q|\alpha\w|\leq m}\CjN{\q(\delta X \w)^{\alpha}
c_{x,j}^{\delta}}{0}{\B{X}{d}{x}{\tau_1\delta}}<\infty$, for every
$m\in \N$.\footnote{For an arbitrary set $U\subseteq \R^n$, we define $\CjN{f}{0}{U}=\sup_{y\in U} \q|f\q(u\w)\w|$, and if we say $\CjN{f}{0}{U}<\infty$, we mean that $f$ is continuous on $U$ and the norm is finite.}
\end{itemize}
Note that, since $\tau_1$ and $U$ may be chosen as small as we
wish, this is a local property.
\end{defn}

\begin{defn}
Let $\sS$ be a, possibly infinite, set of germs of $C^\infty$ vector
fields, $X$, defined on a neighborhood of $x_0\in \R^n$ each paired
with a nonzero formal degree $0\ne d\in \N^{\nu}$.  Let $\q(
X_0,d_0\w)$ be another germ of a $C^{\infty}$ vector field defined
on a neighborhood of $x_0$, with formal degree $0\ne d_0\in
\N^{\nu}$. We say $\sS$ {\it controls} $\q( X_0,d_0\w)$ on a
neighborhood of $x_0$ if there is a finite subset $\sF\subseteq \sS$
such that $\sF$ controls $\q( X_0,d_0\w)$ on a neighborhood of $x_0$
in the sense of Definition \ref{DefnFiniteSetControl}.
\end{defn}

\begin{rmk}
Much more detailed information on this notion of control can be
found in Section 5.3 of \cite{StreetMultiParameterCCBalls} and
Section 11.1 of \cite{StreetMultiParameterSingRadonLt}.
\end{rmk}

\begin{defn}
Let $\sS$ be a set of germs of $C^\infty$ vector fields defined
on a neighborhood of $x_0\in \R^n$ each paired with a $\nu$-parameter
formal degree $0\ne d\in \N^\nu$.  We define $\sL\q(\sS\w)$
to be the smallest set of germs of vector fields with formal degrees
such that:
\begin{itemize}
\item $\sS\subseteq \sL\q(\sS\w)$,
\item if $\q(X_1,d_1\w),\q(X_2,d_2\w)\in \sL\q(\sS\w)$ then $\q( \q[X_1,X_2\w],d_1+d_2\w)\in \sL\q(\sS\w)$.
\end{itemize}
Furthermore, define $\sL_0\q(\sS\w)$ to be the smallest set of
germs of vector fields with formal degrees such that:
\begin{itemize}
\item $\sS\subseteq \sL_0\q(\sS\w)$,
\item if $\q(X_1,d_1\w)\in \sS$ and $\q(X_2,d_2\w)\in \sL_0\q(\sS\w)$, then
$\q( \q[X_1,X_2\w],d_1+d_2\w)\in \sL_0\q(\sS\w)$.
\end{itemize}
\end{defn}

\begin{rmk}\label{RmkJacobi}
Note, by the Jacobi identity, for every $\q(Y_0,d_0\w)\in \sL\q(\sS\w)$,
\begin{equation*}
Y_0\in \Span{Y : \q(Y,d_0\w)\in \sL_0\q(\sS\w) }.
\end{equation*}
\end{rmk}

%% file: results.tex
We begin by rigorously stating our maximal result.  Let $\gamma\q(
t,x\w) = \gamma_t\q(x\w): \R_0^N\times \R_0^n\rightarrow \R^n$ be a
real analytic function defined on a neighborhood of
$\q(0,0\w)\in \R^N\times \R^n$, satisfying $\gamma_0\q( x\w) \equiv
x$.

For $\psi\in C_0^\infty\q( \R^n\w)$ supported on a sufficiently
small neighborhood of $0$, $\psi\geq 0$, and $a>0$ sufficiently
small, define,
\begin{equation*}
\sM f\q( x\w) = \sup_{0<\delta_1,\ldots, \delta_N\leq 1} \psi\q(
x\w) \int_{\q|t\w|<a} \q|f\q( \gamma_{\delta_1 t_1,\ldots, \delta_N
t_N}\q( x\w) \w)\w| \: dt.
\end{equation*}
\begin{thm}\label{ThmRealAnalMax}
$\sM$ is bounded $L^p\rightarrow L^p$ ($1<p\leq\infty$), provided
$a$ is taken sufficiently small, and $\psi$ is supported on a
sufficiently small neighborhood of $0$.
\end{thm}
Theorem \ref{ThmRealAnalMax} will follow from a more general maximal
theorem about $C^\infty$ $\gamma$ which is discussed in
Section \ref{SectionMoreGenMax}. This theorem will imply maximal
results for even stronger maximal functions than are covered by
Theorem \ref{ThmRealAnalMax}.

Fix $\nu$-parameter dilations $e=\q( e_1,\ldots, e_N\w)$ on $\R^N$,
as in Section \ref{SectionKernels} (so that $0\ne e_j\in \N^{\nu}$).
For a multi-index $\alpha=\q(\alpha_1,\ldots, \alpha_N\w)\in
\N^{N}$, define,
\begin{equation*}
\deg\q( \alpha\w) = \sum_{j=1}^N e_j \alpha_j\in \N^{\nu}.
\end{equation*}

\begin{defn}
We call $\alpha$ a {\it pure power} if $\deg\q( \alpha\w)$ is
nonzero in precisely one component.  Otherwise we call $\alpha$ a
{\it non-pure power}.
\end{defn}

Let $\gamma$ be as above.  As discussed in the introduction,
\cite{ChristNagelSteinWaingerSingularAndMaximalRadonTransforms}
showed that $\gamma$ could be written asymptotically as,
\begin{equation*}
\gamma_t\q(x\w) \sim \exp\q(\sum_{0<\q|\alpha\w|} t^{\alpha}
X_\alpha\w)x,
\end{equation*}
where the $X_\alpha$ are real analytic vector fields. Define two
sets,
\begin{equation}\label{EqnDefnsP}
\begin{split}
\sP &= \q\{\q(X_\alpha,\deg\q(\alpha\w)\w): \alpha \text{ is a pure power}\w\},\\
\sN &= \q\{\q(X_\alpha,\deg\q(\alpha\w)\w): \alpha \text{ is a
non-pure power}\w\}.
\end{split}
\end{equation}


\begin{thm}\label{ThmMainSingIntThm}
Suppose that for every $\q( X,d\w) \in \sN$, $\sL\q(\sP\w)$ controls $\q(
X,d\w)$ on a neighborhood of $0$.\footnote{The particular neighborhood used
in the definition of control may depend on $\q(X,d\w)$.}  
Then, there exists $a>0$ such
that for every $\psi_1,\psi_2\in C_0^\infty\q(\R^n \w)$ 
supported on a sufficiently small neighborhood of $0$,
every $K\in \sK\q( N,e,a,\nu\w)$, and
every $C^\infty$ function $\kappa\q( t,x\w)$,
 the operator given by
\begin{equation}\label{EqnSingIntThmOp}
T f\q( x\w) = \psi_1\q( x\w) \int f\q( \gamma_t\q( x\w) \w)
\psi_2\q( \gamma_t\q( x\w)\w) \kappa\q( t,x\w)K\q( t\w) \: dt
\end{equation}
is bounded $L^p\rightarrow L^p$ ($1<p<\infty$).
\end{thm}

\begin{rmk}
Note, by taking $\psi_2$ to be equal to $1$ on a neighborhood of the
support of $\psi_1$, taking $\kappa=1$, and taking $a>0$ so small
that for $t$ in the support of $K\q( t\w)$ and $x$ in the support of
$\psi_1\q( x\w)$ we have $\psi_2\q( \gamma_t\q( x\w)\w)=1$, we see
that the operator given by \eqref{EqnIntroOpForm} is of the form
discussed in Theorem \ref{ThmMainSingIntThm}.
\end{rmk}

Theorem \ref{ThmMainSingIntThm} will follow from a more general
theorem about $C^\infty$ $\gamma$, which is proven in
\cite{SteinStreetMultiParameterSingRadonLp}.

\begin{cor}\label{CorAllPure}
Suppose that $X_\alpha=0$ for every non-pure power $\alpha$.  Then
the operator given by \eqref{EqnSingIntThmOp} is bounded on $L^p$
($1<p<\infty$).
\end{cor}
\begin{proof}
It follows immediately from the definitions that $\sL\q(\sP\w)$ controls
$\q(0,\deg\q( \alpha\w)\w)$ for every $\alpha$.  Thus the hypotheses
of Theorem \ref{ThmMainSingIntThm} hold trivially.
\end{proof}

\begin{cor}\label{CorSingleParamSingInt}
In the special case $\nu=1$ (i.e., the single-parameter case, when
$K\q( t\w)$ is a Calder\'on-Zygmund kernel), the operator given by
\eqref{EqnSingIntThmOp} is bounded on $L^p$ ($1<p<\infty$).
\end{cor}
\begin{proof}
In the single-parameter case, every $\alpha$ is a pure power.  Thus,
the hypotheses of Corollary \ref{CorAllPure} hold vacuously in this
case.
\end{proof}

\begin{prop}\label{PropClosedUnderComp}
Suppose that $T_1$ and $T_2$ are operators of the form covered in
Theorem \ref{ThmMainSingIntThm}:
\begin{equation*}
T_j f\q( x\w) = \psi_1^j\q( x\w) \int f\q( \gamma_{t_j}^j\q( x\w)
\w) \psi_2^j\q( \gamma_{t_j}^j\q( x\w)\w) \kappa_j\q( t,x\w)K_j\q(
t\w) \: dt_j,\quad j=1,2,
\end{equation*}
with $T_j$ satisfying all of the hypotheses of Theorem
\ref{ThmMainSingIntThm} (with perhaps different dilations $e$ for
$\gamma^1_{t_1}$ and $\gamma^2_{t_2}$).
 Then, $T_1T_2$ and $T_1^{*}$ satisfy the hypotheses Theorem \ref{ThmMainSingIntThm} (provided the $K_j$ and
$\psi_1^j,\psi_2^j$ have sufficiently small support).
\end{prop}

Of course, Proposition \ref{PropClosedUnderComp} does not
lead to any new $L^p$ boundedness results:  since $T_1$ and $T_2$
are bounded on $L^p$ ($1<p<\infty$) the same is true
for $T_1T_2$ (similarly for $T_1^{*}$).  
What Proposition \ref{PropClosedUnderComp} does tell us is that
our assumptions are robust enough that one cannot use our theorem,
plus algebraic manipulation of the operators in question, to create
new operators which are bounded on $L^p$, but to which our theorem
does not apply.
Proposition \ref{PropClosedUnderComp} is proved at the end of
Section \ref{SectionReductionSingInt}.

%% file: pastwork.tex
There are a number of papers concerning singular and maximal Radon
transforms.  We review in this section a few which are closely
related to our results.  The results mentioned here served as
motivation for the results in this paper.

One of the first results which comes to mind when considering
maximal Radon transforms associated to real analytic curves is the
following result of Bourgain
\cite{BourgainARemarkOnTheMaximalFunctionAssociatedToAnAnalyticVectorField}:
Theorem \ref{ThmRealAnalMax} holds in the special case when
$\gamma:\R_0\times \R_0^2\rightarrow \R^2$ is given by $\gamma_t\q(
x\w) = x+tv\q( x\w)$ and $v$ is a germ of a real analytic vector
field on $\R^2$.  In a manner completely analogous to this paper,
Bourgain proves a more general maximal theorem about $C^\infty$
curves.  This more general maximal theorem can be seen as a special
case of the maximal result in
\cite{SteinStreetMultiParameterSingRadonLp}.

The paper which served as the primary motivation for the methods in
\cite{StreetMultiParameterSingRadonLt,SteinStreetMultiParameterSingRadonLp}
was due to Christ, Nagel, Stein, and Wainger
\cite{ChristNagelSteinWaingerSingularAndMaximalRadonTransforms},
which discussed the single-parameter case (i.e., when $\nu=1$). As
discussed in \cite{StreetMultiParameterSingRadonLt}, the methods in
\cite{ChristNagelSteinWaingerSingularAndMaximalRadonTransforms} are
not sufficient to obtain Corollary \ref{CorSingleParamSingInt}.
Nevertheless, Christ, Nagel, Stein, and Wainger were able to obtain
a differentiation theorem for real analytic $\gamma$, see
Section 21 of
\cite{ChristNagelSteinWaingerSingularAndMaximalRadonTransforms}.
Namely, for any $f\in L^p$ ($1<p\leq\infty$) supported sufficiently
close to $0$,
\begin{equation*}
f\q( x\w) = \lim_{r\rightarrow 0} c_N^{-1} r^{-N} \int_{\q|t\w|\leq
r} f\q( \gamma_t\q( x\w) \w) \: dt, \quad \text{a.e.}
\end{equation*}
 As is well known, this follows from Theorem \ref{ThmRealAnalMax}.
In fact, it follows from the weaker result where one takes the
supremum over all $\delta_1=\delta_2=\cdots=\delta_N$.

In fact, the basic idea of the proof of the differentiation theorem
in \cite{ChristNagelSteinWaingerSingularAndMaximalRadonTransforms}
is closely related to the results in this paper.  Indeed, the result
follows by applying the Frobenius theorem to show that the ambient
space is foliated into leaves, and other results from
\cite{ChristNagelSteinWaingerSingularAndMaximalRadonTransforms}
could be applied to each leaf to obtain the differentiation theorem.
The main reason why our results are stronger than those in
\cite{ChristNagelSteinWaingerSingularAndMaximalRadonTransforms} in
the single-parameter case, is that we have access to a stronger form
of the Frobenius theorem:  the one developed in
\cite{StreetMultiParameterCCBalls}.

The last paper we wish to mention is due to Christ
\cite{ChristTheStrongMaximalFunctionOnANilpotentGroup}; in it, the
``strong maximal function associated to a nilpotent Lie group'' is
discussed.  Let $G$ be a connected, simply connected, nilpotent Lie
group.  Let $X_1,\ldots, X_N$ be left invariant vector fields on
$G$. Define $\gamma:\R^N\times G\rightarrow \R^N$ by
\begin{equation*}
\gamma_{t_1,\ldots, t_N}\q( x\w) = e^{t_1 X_1+\cdots + t_N X_N}x.
\end{equation*}
Note that we may choose coordinates so that $\gamma$ is real
analytic. Define a maximal function by,
\begin{equation*}
\sMh f\q( x\w) = \sup_{0<\delta_1,\ldots, \delta_N}
\int_{\q|t\w|\leq 1} \q|f\q( \gamma_{\delta_1 t_1,\ldots, \delta_N
t_N}\q( x\w) \w)\w|\: dt.
\end{equation*}
It is shown in
\cite{ChristTheStrongMaximalFunctionOnANilpotentGroup} that $\sMh$
is bounded on $L^p$ ($1<p\leq \infty$).  There are a few differences
between $\sMh$ and the maximal operator discussed in Theorem
\ref{ThmRealAnalMax}.  First, $\sMh$ does not involve a cutoff
function $\psi$; this is due to the translation invariance of
$\sMh$, and is not an essential point.  Second, the supremum in the
definition of $\sMh$ is over all $\delta_1,\ldots, \delta_N$, while
in Theorem \ref{ThmRealAnalMax} we restrict attention to
$\delta_1,\ldots, \delta_N$ small.  The reason the results in
\cite{ChristTheStrongMaximalFunctionOnANilpotentGroup} can be stated
for all $\delta$ is that they are lifted to a setting where there
exist global dilations so that the result for all $\delta$ follows
from the result for small $\delta$; and so this is not an essential
point either.  Thus, the $L^p$ boundedness of $\sMh$ is essentially
a special case of Theorem \ref{ThmRealAnalMax}.

In fact, \cite{ChristTheStrongMaximalFunctionOnANilpotentGroup}
studies even stronger maximal functions than $\sMh$.  While these
are not a special case of Theorem \ref{ThmRealAnalMax}, they are a
special case of the maximal function discussed in Section
\ref{SectionMoreGenMax}.  Thus, the results in
\cite{ChristTheStrongMaximalFunctionOnANilpotentGroup} are a special
case of the results in this paper.  Moreover, the methods in
\cite{ChristTheStrongMaximalFunctionOnANilpotentGroup} provided the
main motivation for the results in Section \ref{SectionMoreGenMax}.

%% file: moregenmax.tex
In this section, we introduce a stronger maximal theorem.  In
Section \ref{SectionReductionMaxInt}, we will show that this maximal
theorem implies Theorem \ref{ThmRealAnalMax}.

Before we introduce this maximal theorem, we must explain the
connection between germs of $C^\infty$ functions satisfying
$\gamma_0\q( x\w) \equiv x$, and certain vector fields.  Indeed,
given a germ of a $C^{\infty}$ function $\gamma_t\q( x\w)$, defined on
a neighborhood of $\q(0,0\w) \in \R^N\times \R^n$, and satisfying
$\gamma_0\q( x\w)\equiv x$, it makes sense to consider
$\gamma_t^{-1}\q( x\w)$, since for $t$ sufficiently small,
$\gamma_t$ is a diffeomorphism onto its image.

Thus, we may define,
\begin{equation*}
W\q( t,x\w) = \frac{d}{d\epsilon}\bigg|_{\epsilon=1}
\gamma_{\epsilon t}\circ \gamma_{t}^{-1}\q( x\w) \in T_x\R^n.
\end{equation*}
Note that $W\q( t\w)$ is vector field, depending smoothly on $t$
such that $W\q( 0\w) \equiv 0$.

\begin{prop}[Proposition 12.1 of
\cite{StreetMultiParameterSingRadonLt}]\label{PropBijectGammaW}
 The map $\gamma\mapsto W$ is
a bijection between germs of $C^{\infty}$ functions, as above, to germs
of vector fields, $W\q( t\w)$, depending smoothly on $t$ and
satisfying $W\q( 0\w)\equiv 0$.
\end{prop}
\begin{proof}[Proof sketch]
The inverse of the map $\gamma\mapsto W$ is as follows.  Given $W$,
let $\omega\q( \epsilon, t,x\w)$ be the unique solution to the ODE:
\begin{equation*}
\frac{d}{d\epsilon}\omega\q( \epsilon,t,x\w) = \frac{1}{\epsilon}
W\q( \epsilon t, \omega\q( \epsilon,t,x\w)\w), \quad \omega\q(
0,t,x\w) = x.
\end{equation*}
Define $\gamma_t\q( x\w) = \omega\q( 1,t,x\w)$.\footnote{It is easy
to see, via the contraction mapping principle, that the solution
$\omega$ exists up to $\epsilon=1$ for $t$ sufficiently small.} This
map $W\mapsto \gamma$ is the two-sided inverse to the map
$\gamma\mapsto W$.  See Proposition 12.1 of
\cite{StreetMultiParameterSingRadonLt} for details.
\end{proof}

In light of Proposition \ref{PropBijectGammaW}, instead of defining
$\gamma$, we may instead define $W$.  This will allow us to
introduce dilations on $\gamma_t$ that are not of the form
$\gamma_{\q(\delta_1 t_1,\ldots, \delta_N t_N\w)}$, thereby allowing
us to introduce stronger maximal functions than are covered in
Theorem \ref{ThmRealAnalMax}. A similar idea was used in
\cite{ChristTheStrongMaximalFunctionOnANilpotentGroup}, though the
setting was simpler and the vector field $W$ did not need to be
introduced.

We now turn to defining the maximal function.  Let $\q( X,d\w) =\q(
X_1,d_1\w),\ldots, \q( X_q,d_q\w)$ be germs of $C^{\infty}$ vector
fields defined on a neighborhood of $0\in \R^n$, each with an
associated formal degree $0\ne d_j\in \N^\nu$.  We suppose that
$\q(\q[X_j,X_k\w], d_j+d_k\w)$ is controlled by $\q(X,d\w)$ on a
neighborhood of $0$, for every $1\leq j,k\leq q$.  Let $1\leq r\leq
q$, and suppose each $d_j$, $1\leq j\leq r$ is nonzero in only one
component.  Suppose further that $\q( X_1,d_1\w),\ldots, \q(
X_r,d_r\w)$ generate $\q( X_{r+1}, d_{r+1}\w),\ldots, \q(
X_q,d_q\w)$ in the sense that 
$\q(X_{r+1},d_{r+1}\w),\ldots, \q(X_q,d_q\w)\in \sL_0\q( \q\{\q(X_1,d_1\w),\ldots, \q(X_r,d_r\w)\w\}\w)$.

We suppose we are given $\nu$-parameter dilations $e=\q( e_1,\ldots,
e_N\w)$ on $\R^N$, as in Section \ref{SectionKernels}; thus it makes
sense to write $\delta t$ for $t\in \R^N$ and $\delta\in
\q[0,1\w]^\nu$ (see \eqref{EqnDefndeltat}).

We suppose we are given germs of $C^\infty$ functions,
\begin{equation*}
c_j\q( t,s,x\w):\R_0^N\times \R_0^N\times \R_0^n\rightarrow \R,\quad
j=1,\ldots, q,
\end{equation*}
with $c_j\q(0,0,x\w) \equiv 0$. Suppose $0\ne\alpha_1,\ldots,
\alpha_r\in \N^{\nu}$ be multi-indices such that,
\begin{equation}\label{EqnGenMaxCoefNonVanish}
\frac{1}{\alpha_l!} \frac{\partial}{\partial s}^{\alpha_l} c_l\q(
t,s,x\w)\bigg|_{t=s=0} =1,\quad 1\leq l\leq r,
\end{equation}
and for all $1\leq l\leq r$, $1\leq k\leq q$ and all
$\beta_1,\beta_2$ with $\beta_1+\beta_2=\alpha_l$,
\begin{equation}\label{EqnGenMaxCoefVanish}
\frac{\partial}{\partial t}^{\beta_1} \frac{\partial}{\partial
s}^{\beta_2} c_k\q( t,s,x\w)\bigg|_{t=s=0} =0, \quad \text{unless
}l=k, \beta_1=0, \beta_2=\alpha_l.
\end{equation}

Let $\Ninf=\N\cup \q\{\infty\w\}$.  For each $j\in \Ninf^\nu$,
define,
\begin{equation*}
W_j\q( t,x\w)=\sum_{l=1}^q c_l\q( 2^{-j} t,t,x\w) 2^{-j\cdot d_l}
X_l.
\end{equation*}
Given $W_j$ we obtain a corresponding $\gamma^j_t$ as in
Proposition \ref{PropBijectGammaW}.  That is, let $\omega_j$ be the
unique solution to the ODE:
\begin{equation*}
\frac{d}{d\epsilon}\omega_j\q( \epsilon,t,x\w) = \frac{1}{\epsilon}
W_j\q( \epsilon t, \omega_j\q( \epsilon,t,x\w)\w), \quad \omega_j\q(
0,t,x\w) = x.
\end{equation*}
Set $\gamma^j_t\q( x\w) = \omega_j\q( 1,t,x\w)$.  It is easy to see,
via the contraction mapping principle, that there are open sets
$0\in U\subset \R^N$, $0\in V\subset \R^n$, independent of $j$, such
that $\gamma_t^j: U\times V\rightarrow \R^n$.

Let $\psi_1,\psi_2\in C_0^{\infty}\q( \R^n\w)$ be supported on a
small neighborhood of $0$, $\psi_1,\psi_2\geq 0$, and let $a>0$ be a
small number.  In light of the above remarks, it makes sense to
define the maximal function,
\begin{equation*}
\sMt f\q( x\w) = \sup_{j\in \N^{\nu}}\psi_1\q( x\w) \int_{\q|t\w|<a}
\q|f\q( \gamma_t^j\q( x\w) \w)\w| \psi_2\q( \gamma_t^j\q( x\w)\w)\:
dt.
\end{equation*}

\begin{thm}\label{ThmMoreGenMaxFunc}
Under the above conditions $\sMt$ is bounded on $L^p$ ($1<p\leq
\infty$), provided $\psi_1$ and $\psi_2$ are supported on a
sufficiently small neighborhood of $0$, and $a>0$ is sufficiently
small.
\end{thm}

Theorem \ref{ThmMoreGenMaxFunc} is proved in Section
\ref{SectionProofOfMax}.

\begin{rmk}
We will see that Theorem \ref{ThmRealAnalMax} follows from Theorem
\ref{ThmMoreGenMaxFunc}.  It is not hard to see that Theorem 2.4 of
\cite{ChristTheStrongMaximalFunctionOnANilpotentGroup} follows from
Theorem \ref{ThmMoreGenMaxFunc}.  It follows that all of the results
of \cite{ChristTheStrongMaximalFunctionOnANilpotentGroup} can be
reduced to Theorem \ref{ThmMoreGenMaxFunc}.
\end{rmk}

\begin{rmk}
Notice that we have discretized our maximal functions; i.e.,
we only consider dyadic scales.  This is
essential when considering $\sMt$.  Indeed, the obvious non-discretized
version need not be bounded on all $L^p$, $p>1$.
This was noted on the top of page 5 of \cite{ChristTheStrongMaximalFunctionOnANilpotentGroup}.
\end{rmk}

%% file: lpres.tex
In this section, we review the results of \cite{SteinStreetMultiParameterSingRadonLp}.
We will see that Theorem \ref{ThmMainSingIntThm} is, in fact,
a special case of Theorem 5.2 of \cite{SteinStreetMultiParameterSingRadonLp},
which we review below.  In addition, we rephrase the assumptions
of \cite{SteinStreetMultiParameterSingRadonLp} in a few different ways,
which will be useful in what follows.

The setting is as follows.  We are given a $C^\infty$ function
$\gamma_t\q(x\w)=\gamma\q(t,x\w):\R^N_0\times \R^n_0\rightarrow \R^n$
satisfying $\gamma_0\q(x\w)\equiv x$.
The goal is to give conditions on $\gamma$ such that the operator given by\footnote{Or more generally, operators of the form covered in Theorem \ref{ThmMainSingIntThm}.}
\begin{equation}\label{EqnToStateBoundedLp}
Tf\q(x\w) = \psi\q(x\w)\int f\q(\gamma_t\q(x\w)\w) K\q(t\w)\: dt,
\end{equation}
is bounded on $L^p$ ($1<p<\infty$) for every $K\in \sK\q(N,e,a,\nu\w)$
where $a$ is sufficiently small and $\psi$ is supported on a sufficiently
small neighborhood of $0$.  We think of the $\nu$-parameter
dilations $e$ as fixed so that it makes sense to write
$\delta t$ for $\delta\in \q[0,\infty\w)^\nu$ and $t\in \R^N$
as in \eqref{EqnDefndeltat}.

\begin{defn}\label{DefnControlW}
Let $\q(X,d\w)=\q(X_1,d_1\w),\ldots, \q(X_q,d_q\w)$ be a finite list of 
$C^\infty$ vector fields with $\nu$-parameter formal degrees
$0\ne d_j\in \q[0,\infty\w)^\nu$ as in Section \ref{SectionCC}.
Let $W\q(t,x\w)\in T_x\R^n$ be a smooth vector field (defined on a neighborhood
of $\q(0,0\w)\in \R^N\times \R^n$), depending
smoothly on $t\in \R^N_0$.  We say that $\q(X,d\w)$ {\it controls} $W$
on a neighborhood of $0$ if there exists an open neighborhood $U$ of $0\in \R^n$, $\tau_1>0$, and $\rho_1>0$, such that for every $\delta\in \q[0,1\w]^\nu$, $x_0\in U$,
there exist functions $c_l^{x_0,\delta}$ on $B^N\q(\rho_1\w)\times \B{X}{d}{x_0}{\tau_1\delta}$ satisfying
\begin{itemize}
\item $W\q(\delta t, x\w) = \sum_{l=1}^q c_l^{x_0,\delta}\q(t,x\w) \delta^{d_l} X_l\q(x\w)$ on $B^N\q(\rho_0\w)\times \B{X}{d}{x_0}{\tau_1\delta}$.
\item $\sup_{\substack{x_0\in U\\ \delta\in \q[0,1\w]^\nu}} \sum_{\q|\alpha\w|+\q|\beta\w|\leq m} \CjN{\q(\delta X\w)^\alpha \partial_t^\beta c_l^{x_0,\delta}}{0}{B^N\q(\rho_1\w)\times \B{X}{d}{x_0}{\tau_1\delta}}<\infty$, for every $m$.
\end{itemize}
If, instead, $\sS$ is an infinite collection of vector fields,
then we say $\sS$ {\it controls} $W$ on a neighborhood of $0$ if there is a finite subset
which controls $W$.
\end{defn}

\begin{rmk}\label{RmkContrlWControlX}
Definition \ref{DefnControlW} is closely related
to Definition \ref{DefnFiniteSetControl}.  Indeed,
note that if $\q(X,d\w)$ controls $W$ on a neighborhood of $0$, and if the Taylor series
for $W$ is given by,
\begin{equation*}
W\q(t,x\w)\sim \sum_{\alpha} t^{\alpha} Y_\alpha,
\end{equation*}
then $\q(X,d\w)$ controls $\q(Y_\alpha, \deg\q(\alpha\w)\w)$ on a neighborhood
of $0$ for every
$\alpha$.
\end{rmk}

\begin{defn}\label{DefnGeneratesAFiniteList}
Given a finite list of $C^\infty$ vector fields (defined on a neighborhood of $0\in \R^n$) with $\nu$-parameter formal
degrees $\q(X_1,d_1\w),\ldots, \q(X_r,d_r\w)$ we say that
this list {\it generates the finite list} $\q(X,d\w)=\q(X_1,d_1\w),\ldots, \q(X_q,d_q\w)$ (here $q\geq r$) if there exist vector fields with formal 
degrees $\q(X_{r+1},d_{r+1}\w),\ldots, \q(X_q,d_q\w)\in \sL_0\q(\q\{\q(X_1,d_1\w),\ldots, \q(X_r,d_r\w)\w\}\w)$ such that 
for every $1\leq j,k\leq q$, $\q(\q[X_j,X_k\w],d_j+d_k\w)$ is controlled by $\q(X,d\w)$ on a neighborhood of $0$.
\end{defn}

\begin{rmk}
In what follows, we will say $A$ {\it controls} $B$ to mean
$A$ controls $B$ on a neighborhood of $0$.
\end{rmk}

With the above definitions in hand, we are prepared to state the assumptions
placed on $\gamma$ in \cite{SteinStreetMultiParameterSingRadonLp}.
We state these assumptions in three different ways, which we will
see are all equivalent.  Under any of the following assumptions,
the operator given by \eqref{EqnToStateBoundedLp} is bounded on
$L^p$ ($1<p<\infty$).
In what follows, define the vector field,
\begin{equation*}
W\q(t,x\w) = \frac{d}{d\epsilon}\bigg|_{\epsilon=1} \gamma_{\epsilon t}\circ \gamma_t^{-1}\q(x\w)\in T_x\R^n.
\end{equation*}
Note that $\gamma_t^{-1}$ makes sense, since for $t$ sufficiently small,
$\gamma_t$ is a diffeomorhpism onto its image (because $\gamma_0\q(x\w)\equiv x$).

\begin{enumerate}[(I)]
\item\label{ItemSSCond} Expand $W$ as a Taylor series in the $t$ variable,
\begin{equation*}
W\q(t\w)\sim \sum_{\q|\alpha\w|>0} t^{\alpha} \Xh_\alpha,
\end{equation*}
where the $\Xh_\alpha$ are $C^\infty$ vector fields.
We assume that there is a finite subset,
\begin{equation*}
\sF\subseteq \q\{\q(\Xh_\alpha, \deg\q(\alpha\w)\w) : \alpha\text{ is a pure power}\w\},
\end{equation*}
such that $\sF$ generates a finite list $\q(\Xh,d\w)$ and $\q(\Xh,d\w)$
controls $W$.

\item\label{ItemCondWithW} For the second equivalent condition, we rephrase (\ref{ItemSSCond}) as having
two distinct parts:
\begin{itemize}
\item[(\ref{ItemCondWithW}.F)]\label{ItemCondWithWa} A ``finite type'' condition:  taking $\Xh_{\alpha}$ as in (\ref{ItemSSCond}), we assume that there is a finite subset,
\begin{equation*}
\sF\subseteq \q\{\q(\Xh_\alpha,\deg\q(\alpha\w)\w): \alpha\in \N^N \w\},
\end{equation*}
such that $\sF$ generates a finite list $\q(\Xh,d\w)$ and this finite list controls $W$.
\item[(\ref{ItemCondWithW}.A)]\label{ItemCondWithWb} An ``algebraic'' condition:  we assume that for every non-pure power
$\alpha$, $\q(\Xh_\alpha, \deg\q(\alpha\w)\w)$ is controlled by
$$\sL\q(\q\{\q(\Xh_\alpha, \deg\q(\alpha\w)\w):\alpha\text{ is a pure power}\w\}\w).$$
\end{itemize}

\item\label{ItemCondFinal}  The third equivalent condition is the same as (\ref{ItemCondWithW}), except we use different vector fields.  Indeed, write,
\begin{equation*}
\gamma_t\q(x\w)\sim\exp\q(\sum_{\alpha} t^{\alpha} X_\alpha\w)x.
\end{equation*}
\begin{itemize}
\item[(\ref{ItemCondFinal}.F)] A ``finite type'' condition:  we assume there is a finite subset,
\begin{equation*}
\sF\subseteq \q\{\q(X_\alpha, \deg\q(\alpha\w)\w):\alpha\in \N^\nu\w\},
\end{equation*}
such that $\sF$ generates a finite list $\q(X,d\w)$ and this finite list
controls $W$.
\item[(\ref{ItemCondFinal}.A)] An ``algebraic'' condition:  we assume that for every non-pure power
$\alpha$, $\q( X_\alpha, \deg\q(\alpha\w)\w)$ is controlled by
$$\sL\q(\q\{\q(X_\alpha, \deg\q(\alpha\w)\w):\alpha\text{ is a pure power}\w\}\w).$$
\end{itemize}
\end{enumerate}

\begin{rmk}
Note that the assumptions of Theorem \ref{ThmMainSingIntThm}
are exactly that (\ref{ItemCondFinal}.A) holds.
\end{rmk}

\begin{rmk}
The vector fields $X_\alpha$ and $\Xh_\alpha$ are closely related.
See Lemma \ref{LemmaCampbellHausdoffSpan}.
\end{rmk}

\begin{thm}\label{ThmEquivConds}
(\ref{ItemSSCond})$\Leftrightarrow$(\ref{ItemCondWithW})$\Leftrightarrow$(\ref{ItemCondFinal}); i.e., the above three conditions are equivalent.
\end{thm}

We prove Theorem \ref{ThmEquivConds} at the end of this section.

\begin{thm}[Theorem 5.2 of \cite{SteinStreetMultiParameterSingRadonLp}]\label{ThmSSLpBoundedness}
Under any of the above three conditions, there exists $a>0$ such that
for every $\psi_1,\psi_2\in C_0^\infty\q(\R^n\w)$ supported on a sufficiently
small neighborhood of $0$, every $K\in \sK\q(N,e,a,\nu\w)$, and every
$C^\infty$ function $\kappa\q(t,x\w)$, the operator given by,
\begin{equation*}
Tf\q(x\w) =\psi_1\q(x\w) \int f\q(\gamma_t\q(x\w)\w) \psi_2\q(\gamma_t\q(x\w)\w) \kappa\q(t,x\w) K\q(t\w)\: dt
\end{equation*}
is bounded $L^p\rightarrow L^p$ ($1<p<\infty$).
\end{thm}
\begin{proof}
Under the assumption (\ref{ItemSSCond}), this is contained in Theorem 5.2
of \cite{SteinStreetMultiParameterSingRadonLp}.
\end{proof}

\begin{prop}\label{PropFiniteTypeAuto}
When $\gamma$ is real analytic, 
(\ref{ItemCondWithW}.F) and (\ref{ItemCondFinal}.F) hold automatically.
\end{prop}

We defer the proof of Proposition \ref{PropFiniteTypeAuto}
to Section \ref{SectionReduction}.  From the above results,
Theorem \ref{ThmMainSingIntThm} follows easily.

\begin{proof}[Proof of Theorem \ref{ThmMainSingIntThm} given the above results]
By Theorem \ref{ThmSSLpBoundedness}, it suffices to show that (under
the assumptions of Theorem \ref{ThmMainSingIntThm}),
(\ref{ItemCondFinal}) holds.
Proposition \ref{PropFiniteTypeAuto} shows that 
(\ref{ItemCondFinal}.F) holds, while the assumptions
of Theorem \ref{ThmMainSingIntThm} are exactly that 
(\ref{ItemCondFinal}.A) holds.
\end{proof}

We close this section by proving Theorem \ref{ThmEquivConds}.
We separate Theorem \ref{ThmEquivConds} into two propositions.

\begin{prop}\label{PropSSCondEquivCondWithW}
(\ref{ItemSSCond})$\Leftrightarrow$(\ref{ItemCondWithW}).
\end{prop}

\begin{prop}\label{PropCondWithWEquivCondFinal}
(\ref{ItemCondWithW})$\Leftrightarrow$(\ref{ItemCondFinal}).
More specifically, 
(\ref{ItemCondWithW}.F)$\Leftrightarrow$(\ref{ItemCondFinal}.F)
and (\ref{ItemCondWithW}.A)$\Leftrightarrow$(\ref{ItemCondFinal}.A).
\end{prop}

\begin{lemma}\label{LemmaControlTransitive}
The notion of control is transitive.  Indeed, if $\sS_1$ and $\sS_2$
are sets of vector fields with $\nu$-parameter formal degrees
such that every element of $\sS_2$ is controlled by $\sS_1$, and if $\sS_2$ controls 
a vector field with formal degree $\q(X,d\w)$, then so does $\sS_1$.
A similar result holds if $\q(X,d\w)$ is replaced by $W\q(t,x\w)$
as in Definition \ref{DefnControlW}.
\end{lemma}
\begin{proof}
This follows immediately from the definitions.
\end{proof}

\begin{lemma}\label{LemmaControlCommutator}
If $\sS_1, \sS_2$ are sets of vector fields with $\nu$-parameter formal degrees,
such that,
\begin{itemize}
\item for every $\q(X_1,d_1\w),\q(X_2,d_2\w)\in \sS_1$, $\q(\q[X_1,X_2\w],d_1+d_2\w)$ is controlled by $\sS_1$,
\item every element of $\sS_2$ is controlled by $\sS_1$.
\end{itemize}
Then, every element of $\sL\q(\sS_2\w)$ is controlled by $\sS_1$.
\end{lemma}
\begin{proof}
This follows immediately from the definitions.
\end{proof}

\begin{rmk}
Note that if $\sF$ is a finite set of vector fields which generates
a finite list $\q(X,d\w)=\q(X_1,d_1\w),\ldots, \q(X_q,d_q\w)$ as in Definition
\ref{DefnGeneratesAFiniteList},
then $\q(X,d\w)$ satisfies the hypotheses of $\sS_1$ in Lemma \ref{LemmaControlCommutator}.
\end{rmk}

\begin{proof}[Proof of Proposition \ref{PropSSCondEquivCondWithW}]
(\ref{ItemSSCond})$\Rightarrow$(\ref{ItemCondWithW}): 
(\ref{ItemCondWithW}.F) follows immediately
from (\ref{ItemSSCond}).  (\ref{ItemCondWithW}.A)
follows from (\ref{ItemSSCond}) via Remark \ref{RmkContrlWControlX}.

(\ref{ItemCondWithW})$\Rightarrow$(\ref{ItemSSCond}):  
Take $\sF$ as in (\ref{ItemCondWithW}) and
let $\sF'$ be a finite list generated by $\sF$ (see Definition \ref{DefnGeneratesAFiniteList}), so that $\sF'$ controls
$W$, and set
$$\sPh=\q\{\q(\Xh_\alpha,\deg\q(\alpha\w): \alpha \text{ is a pure power}\w)\w\}.$$
By our assumption, every element of $\sF$
is controlled by $\sL\q(\sPh\w)$.  
Thus, every element of $\sL\q(\sF\w)$ is controlled by $\sL\q(\sPh\w)$ (Lemma
\ref{LemmaControlCommutator}).
It follows that
every element of $\sF'$ is controlled by $\sL\q(\sPh\w)$.
By Remark \ref{RmkJacobi}, every element of $\sF'$ is therefore
controlled by $\sL_0\q(\sPh\w)$.
Let $\sF_0\subseteq \sL_0\q(\sPh\w)$ be a finite subset such that every
element of $\sF'$ is controlled by $\sF_0$.
We may assume that $\sF_0\subseteq \sL_0\q(\sF_0\cap \sPh\w)$; indeed,
since $\sF_0\subseteq \sL_0\q(\sPh\w)$, we may add a finite number of elements
to $\sF_0$ from $\sPh$ so that $\sF_0\subseteq \sL_0\q(\sF_0\cap \sPh\w)$.
Since $\sF'$ controls $W$ on a neighborhood of $0$, it follows
that $\sF_0$ controls $W$ (Lemma \ref{LemmaControlTransitive}).

To complete the proof, we need to show that if $\q(Y_1,d_1\w),\q(Y_2,d_2\w)\in \sF_0$, then $\q(\q[Y_1,Y_2\w],d_1+d_2\w)$ is controlled by $\sF_0$; for then
$\sF_0\cap \sPh$ will generate the finite list $\sF_0$ (which we know to
control $W$).
To do this, it suffices to show that $\q(\q[Y_1,Y_2\w],d_1+d_2\w)$
is controlled by $\sF'$ (by Lemma \ref{LemmaControlTransitive}, since every element of $\sF'$
is controlled by $\sF_0$).
In particular, it suffices to show that every element of
$\sL\q(\sPh\w)$ is controlled by $\sF'$.

By Remark \ref{RmkContrlWControlX}, every element of
$\sPh$ is controlled by $\sF'$.  We know, by assumption, that
if $\q(X_1,d_1\w),\q(X_2,d_2\w)\in \sF'$, then $\q(\q[X_1,X_2\w],d_1+d_2\w)$
is controlled by $\sF'$.  It follows from Lemma \ref{LemmaControlCommutator}
that every element of $\sL\q(\sPh\w)$ is controlled by $\sF'$.
This completes the proof.
\end{proof}

\begin{lemma}\label{LemmaCampbellHausdoffSpan}
Let $\Xh_\alpha$ be as in (\ref{ItemSSCond}) and $X_\alpha$
be as in (\ref{ItemCondFinal}), then, for every $d_0\in \N^\nu$,
\begin{equation*}
\Span{Y : \q(Y,d_0\w)\in \sL\q(\q\{\q(X_\alpha,\deg\q(\alpha\w)\w)\w\}\w) }
=\Span{Y : \q(Y,d_0\w)\in \sL\q(\q\{\q(\Xh_\alpha,\deg\q(\alpha\w)\w)\w\}\w) }
\end{equation*}
\end{lemma}
\begin{proof}
This follows easily from an application of the Campbell-Hausdorff
formula.  See the proof of Proposition 9.6
of \cite{ChristNagelSteinWaingerSingularAndMaximalRadonTransforms}
for a similar result and more details.
\end{proof}

\begin{proof}[Proof of Proposition \ref{PropCondWithWEquivCondFinal}]
We begin by showing 
(\ref{ItemCondWithW}.F)$\Rightarrow$(\ref{ItemCondFinal}.F);
the implication (\ref{ItemCondFinal}.F)$\Rightarrow$(\ref{ItemCondWithW}.F)
follows in the same way, and we leave the details to the reader.
Suppose that (\ref{ItemCondWithW}.F) holds:
there is a finite set $\sF$ as in (\ref{ItemCondWithW}.F) which generates
a finite list, $\sF'$ and this finite list controls $W$.
Define,
\begin{equation*}
\sS=\q\{\q(X_\alpha,\deg\q(\alpha\w)\w): \alpha\in \N\w\}.
\end{equation*}
By Lemma \ref{LemmaCampbellHausdoffSpan}, for every $\q(Y_0,d_0\w)\in \sF'$,
$Y_0\in \Span{Y : \q(Y_0,d_0\w)\in \sL\q(\sS\w) }$.
Thus, there is a finite subset $\sF_0\subseteq \sL\q(\sS\w)$, such that
$\sF_0$ controls $\sF'$.  By Remark \ref{RmkJacobi},
we may assume that $\sF_0\subseteq \sL_0\q(\sS\w)$.  Furthermore,
by adding a finite number of elements to $\sF_0$, we may assume that
$\sF_0\subseteq \sL_0\q( \sS\cap \sF_0\w)$.

Since $\sF'$ controls $W$, we have that $\sF_0$ controls $W$.
To complete the proof, we need only verify that for every
$\q(X_1,d_1\w),\q(X_2,d_2\w)\in \sF_0$, $\q(\q[X_1,X_2\w],d_1+d_2\w)$ is
controlled by $\sF_0$; for then $\sF_0\cap \sS$ will generate the finite list
$\sF_0$ which we already know controls $W$.  That this is true follows just
as in the proof of Proposition \ref{PropSSCondEquivCondWithW}.
This completes the proof of (\ref{ItemCondFinal}.F).

(\ref{ItemCondWithW}.A)$\Leftrightarrow$(\ref{ItemCondFinal}.A) is a simple consequence
of Lemma \ref{LemmaCampbellHausdoffSpan}, which we leave to the reader.
\end{proof}

%% file: realanal.tex
In this section, we introduce the theory we need to see our theorems
concerning real analytic $\gamma$ as special cases of theorems
concerning $C^\infty$ $\gamma$, which are amenable to the methods of
\cite{SteinStreetMultiParameterSingRadonLp}.  Let
\begin{equation*}
\sA_N=\q\{f:\R_0^N\rightarrow \R\big| f\text{ is real analytic}\w\},
\end{equation*}
the set of germs of real analytic functions defined on a
neighborhood of $0\in \R^N$, and taking values in $\R$.  Note that
$\sA_N$ is a ring, and,
\begin{equation*}
\sA_{N}^m=\q\{f:\R_0^N\rightarrow \R^m\big| f\text{ is real
analytic}\w\}.
\end{equation*}
$\sA_N^m$ is an $\sA_N$-module.

\begin{thm}\label{ThmTaylorPrep}
Suppose $f\q(t,x\w):\R_0^N\times\R_0^n\rightarrow \R^m$ is a germ of
a real analytic function; $f\in \sA_{N+n}^m$.  For $\alpha\in
\N^{\nu}$, let $f_\alpha\q( x\w)\in \sA_n^m$ be the Taylor
coefficient of $f$, when the Taylor series is taken in the $t$
variable:
\begin{equation}\label{EqnTaylorPrepTaylorSeries}
f\q( t,x\w) = \sum_{\alpha\in \N^{N}} t^{\alpha} f_\alpha\q( x\w).
\end{equation}
Then, there exist finitely many multi-indices $\alpha_1,\ldots,
\alpha_r\in \N^{\nu}$, and germs of real analytic functions
$c_{\alpha_1},\ldots, c_{\alpha_r}\in \sA_{N+n}$ such that,
\begin{equation}\label{EqnTaylorPrepFiniteSum}
f\q( t,x\w) = \sum_{k=1}^r c_{\alpha_k}\q( t,x\w) t^{\alpha_k}
f_{\alpha_k}\q( x\w),
\end{equation}
on a neighborhood of $\q( 0,0\w)\in \R^N\times \R^n$.  Furthermore,
we may assume for every $1\leq j,k\leq r$,
\begin{equation}\label{EqnTaylorPrepGoodCoef}
\frac{1}{\alpha_j!} \frac{\partial}{\partial t}^{\alpha_j}
t^{\alpha_k} c_{\alpha_k}\q( t,x\w) \bigg|_{t=0} =\begin{cases} 1 &
\text{if }j=k,\\
0 & \text{if }j\ne k.
\end{cases}
\end{equation}
\end{thm}

\begin{thm}\label{ThmFiniteGenControl}
Suppose
$$\sS\subseteq \sA_N^n\times \N^{\nu}.$$
Then there exists a finite subset $\sF\subseteq \sS$ such that 
every $\q( g,e\w)\in \sS$ can be written in the form,
\begin{equation}\label{EqnFiniteGenControl}
g\q(x\w) = \sum_{\substack{\q( f,d\w)\in \sF\\ d\leq e}} c_{\q(
f,d\w)}\q( x\w) f\q( x\w);
\end{equation}
where $c_{\q( f,d\w)}\in \sA_N$, and $d\leq e$ means that the
inequality holds for each coordinate. The neighborhood on which
\eqref{EqnFiniteGenControl} holds may depend on $\q( g,e\w)$.
\end{thm}

\begin{cor}\label{CorNoeFiniteType}
Let $\sS\subseteq \sA_n^n\times \N^{\nu}$. We think of $\sS$ as a
set of pairs $\q( X,d\w)$ where $X$ is the germ of a real analytic
vector field, and $d\in \N^{\nu}$ is a formal degree.
  Then, there exists a
finite subset $\sF\subseteq \sS$ such that every $\q( Y,e\w)\in \sS$
is controlled by $\sF$.
\end{cor}
\begin{proof}
Let $\sF$ be as in the conclusion of Theorem
\ref{ThmFiniteGenControl} when applied to $\sS$.  Let $\q( Y,e\w)\in
\sS$.  We wish to show that $\q( Y,e\w)$ is controlled by $\sF$. For
$\delta\in \q[0,1\w]^\nu$, multiplying both sides of
\eqref{EqnFiniteGenControl} by $\delta^{e}$, we obtain,
\begin{equation*}
\delta^e Y = \sum_{\substack{\q( X,d\w)\in \sF\\ d\leq e}}
\q(\delta^{e-d} c_{\q(X,d\w)}\w) \delta^d X.
\end{equation*}
Noting that $\delta^{e-d} c_{\q(X,d\w) }\in C^\infty$ uniformly for
$\delta\in \q[0,1\w]^\nu$ (since we are only considering the case
when $d\leq e$ coordinatewise), the result follows.
\end{proof}

The above three results are the only results we will need concerning
real analytic functions.  The rest of this section is devoted to
proving and discussing Theorems \ref{ThmTaylorPrep} and
\ref{ThmFiniteGenControl}.  The reader uninterested in the proofs
may safely skip the remainder of this section, as it will not be
used in the sequel.

Theorems \ref{ThmTaylorPrep} and \ref{ThmFiniteGenControl} both
follow easily from well-known results.  We begin by outlining the
results necessary to prove Theorem \ref{ThmFiniteGenControl}.

\begin{prop}[See
\cite{ZariskiSamuelCommutativeAlgebraII}]\label{PropNoeRing}
The
ring $\sA_N$ is Noetherian.
\end{prop}
\begin{proof}[Comments on the proof]
This is a simple consequence of the Weierstrass preparation theorem.
See page 148 of \cite{ZariskiSamuelCommutativeAlgebraII}.  The proof
in \cite{ZariskiSamuelCommutativeAlgebraII} is for the formal power
series ring, however, as mentioned on page 130 of
\cite{ZariskiSamuelCommutativeAlgebraII}, the proof also works for
the ring convergent power series:  i.e., the ring of power series
with some positive radius of convergence.  The ring of germs of real
analytic functions is isomorphic to the ring convergent power
series.
\end{proof}

\begin{prop}\label{PropNoeMod}
The module $\sA_N^m$ is a Noetherian $\sA_N$-module.
\end{prop}
\begin{proof}[Comments on the proof]
It is easy to see that for any Noetherian ring $R$, the $R$-module
$R^m$ is Noetherian.  Actually, in this special case, one can
characterize a finite set of generators for any submodule of
$\sA_N^m$.  This can be found in
\cite{GalligoTheoremeDeDivisionEtStabiliteEnGeometrieAnalytiqueLocale},
but we will not need this.
\end{proof}

\begin{proof}[Proof of Theorem \ref{ThmFiniteGenControl}]
Let $\sS\subseteq \sA_N^n\times \N^{\nu}$.  Define a map $\iota:
\sA_N^n\times \N^{\nu}\rightarrow \sA_{\nu+N}^n$ by,
\begin{equation*}
\iota\q( f,d\w) = t^d f\q( x\w), \quad t\in \R^\nu.
\end{equation*}
Let $M$ be the submodule of $\sA_{\nu+N}^n$ generated by $\iota
\sS$.  $M$ is finitely generated by Proposition \ref{PropNoeMod}.
Let $\sF\subseteq \sS$ be a finite subset such that $\iota \sF$
generates $M$.  We will show that $\sF$ satisfies the conclusions of
the theorem.

Indeed, let $\q( g,e\w)\in \sS$.  Since $t^e g\q( x\w)\in M$, we may
write,
\begin{equation}\label{EqnOneLevelUpNoe}
t^{e} g\q( x\w) = \sum_{\q( f,d\w)\in \sF} \ch_{\q( f,d\w)}\q(
t,x\w) t^{d} f\q(x\w),
\end{equation}
on a neighborhood of $\q(0,0\w)\in \R^{\nu}\times \R^N$.

We apply $\frac{1}{e!}\frac{\partial}{\partial t}^e\big|_{t=0}$ to
both sides of \eqref{EqnOneLevelUpNoe}.  Note that,
\begin{equation*}
\frac{1}{e!}\frac{\partial}{\partial t}^e\bigg|_{t=0}\ch_{\q(
f,d\w)}\q( t,x\w) t^{d} =0, \quad \text{unless } e\geq d.
\end{equation*}
Thus, we obtain,
\begin{equation*}
g\q( x\w) = \sum_{\substack{\q( f,d\w)\in \sF\\ d\leq e}}
\q[\frac{1}{e!}\frac{\partial}{\partial t}^e\bigg|_{t=0}\ch_{\q(
f,d\w)}\q( t,x\w) t^{d}\w] f\q(x\w),
\end{equation*}
completing the proof.
\end{proof}

\begin{rmk}\label{RmkLobry}
The case $\nu=0$ of Corollary \ref{CorNoeFiniteType} in the context
of vector fields seems to have been first used by Lobry
\cite{LobryControlabiliteDesSystemesNonLinearies}.  In
\cite{LobryControlabiliteDesSystemesNonLinearies}, Corollary
\ref{CorNoeFiniteType} was used in the following way.  Let $\sS$ be
a set of germs of real analytic vector fields, and let $\sD$ be the
involutive distribution generated by $\sS$.  In light of the $\nu=0$
case of Corollary \ref{CorNoeFiniteType}, there is a finite subset
$\sF\subseteq \sD$ such that for every $Y\in \sD$, $Y$ can be
written as a sum of elements of $\sF$ (on some suitably small open
set, depending on $Y$). Because of this,
\cite{LobryControlabiliteDesSystemesNonLinearies} said the
distribution $\sD$ was ``locally of finite type.''  
Unfortunately, there is a slight error in the application
the Frobenius theorem in \cite{LobryControlabiliteDesSystemesNonLinearies},
see \cite{StefanIntegrabilityOfSystemsOfVectorFields}--this is due
to the fact that the open set depends on $Y$.
However, in our uses of the Frobenius theorem, we will always 
be able to consider only {\it finite}
sets $\sS$, and then it is easy to see that the open set need not
depend on $Y$.
Corollary \ref{CorNoeFiniteType} can
be considered a ``scale invariant'' version of the ideas of \cite{LobryControlabiliteDesSystemesNonLinearies}.
\end{rmk}

We close this section by proving Theorem \ref{ThmTaylorPrep}.  
We will need a Weierstrass-type preparation theorem
from
\cite{GalligoTheoremeDeDivisionEtStabiliteEnGeometrieAnalytiqueLocale}.
First, we introduce the relevant aspects of
\cite{GalligoTheoremeDeDivisionEtStabiliteEnGeometrieAnalytiqueLocale}
we need (which is only a small fraction of that paper), and then we
will show that Theorem \ref{ThmTaylorPrep} is a simple consequence.
We will only need part of Theorem 1.2.5 of
\cite{GalligoTheoremeDeDivisionEtStabiliteEnGeometrieAnalytiqueLocale},
and we turn to introducing the relevant notation. We will introduce
a division theorem for functions in $\sA_\nu^m$ (of course division
theorems are closely related to preparation theorems). Pick numbers
$\lambda_1,\ldots,\lambda_\nu\in \q(0,\infty\w)$ such that
$\lambda_1,\ldots, \lambda_\nu$ are linearly independent over $\Z$.
For $\alpha=\q(\alpha_1,\ldots, \alpha_\nu\w)\in \N^\nu$, define
$L\q(\alpha\w)=\sum_{j=1}^\nu \alpha_j \lambda_j\in \q[0,\infty\w)$.
Note that if $\alpha\ne \beta$, then $L\q(\alpha\w)\ne
L\q(\beta\w)$. $L$ induces a total ordering on the set
$\N^{\nu}\times \mset$. Indeed, we say
$\q(\alpha,i\w)<\q(\beta,j\w)$ if $L\q(\alpha\w)<L\q(\beta\w)$ or if
$L\q(\alpha\w)=L\q(\beta\w)$ and $i<j$.

For a function $f\in \sA_{\nu}^m$ write $f=\q(f_1,\ldots, f_m\w)$,
where $f_i\in \sA_\nu$.  Write $f_i$ as a Taylor series,
\begin{equation*}
f_i\q(x\w)=\sum_{\alpha\in \N^{\nu}} f_{\alpha,i} x^{\alpha}.
\end{equation*}
Let $Q\q( f\w)$ be the Newton diagram of $f$:
\begin{equation*}
Q\q(f\w)=\q\{\q(\alpha,i\w)\in \N^\nu\times \mset : f_{\alpha,i}\ne
0\w\}.
\end{equation*}
For $f\ne 0$, let $\expL{f}$ to be the smallest element of
$Q\q(f\w)$ in the above defined total ordering. Let $M$ be a
submodule of $\sA_\nu^m$.  We define,
\begin{equation*}
E_L\q(M\w) = \q\{\expL{f} : 0\ne f\in M\w\}.
\end{equation*}

\begin{thm}[Part of Theorem 1.2.5 of
\cite{GalligoTheoremeDeDivisionEtStabiliteEnGeometrieAnalytiqueLocale}]\label{ThmGalligo}
Let $M$ be a submodule of $\sA_\nu^m$.  Then, every element $f\in
\sA_\nu^m$ is congruent, modulo $M$, to a unique element
$r=r\q(f\w)=\q(r_1,\ldots, r_m\w)\in \sA_\nu^m$ of the form,
\begin{equation}\label{EqnGalligoRemainder}
r_i\q(x\w) = \sum_{\q(\alpha,i\w)\not\in E_L\q(M\w)} h_{\alpha,i}
x^{\alpha}.
\end{equation}
That is, the nonzero terms in the Taylor series of $r$ do not appear
in $E_L\q(M\w)$.
\end{thm}

\begin{proof}[Proof of Theorem \ref{ThmTaylorPrep}]
First we prove the result without insisting the $c_{\alpha_k}$
satisfy \eqref{EqnTaylorPrepGoodCoef}.  Then we will show that we
may modify the $c_{\alpha_k}$ so that \eqref{EqnTaylorPrepGoodCoef}
is satisfied.

Express $f$ as a Taylor series as in
\eqref{EqnTaylorPrepTaylorSeries}:
\begin{equation*}
f\q(t,x\w)= \sum_{\alpha\in \N^{N}} t^{\alpha} f_\alpha\q(x\w).
\end{equation*}
Let $M$ be the submodule of $\sA_{N+n}^m$ generated by
$\q\{t^{\alpha} f_\alpha\q(x\w) : \alpha\in \N^N\w\}$.  We know that
$M$ is finitely generated by Proposition \ref{PropNoeMod}, and thus
\eqref{EqnTaylorPrepFiniteSum} will follow if we can show $f\in M$.

Taking the setup of Theorem \ref{ThmGalligo} (and thus we must
choose some $L$), we see that we may write $f$ uniquely modulo $M$
as a term $r$ satisfying \eqref{EqnGalligoRemainder}.  We wish to
show that $r=0$.  Suppose not.  We will show that $\expL{r}\in
E_L\q(M\w)$, which will contradict the form of $r$ given by
\eqref{EqnGalligoRemainder}.

Note that $r=m+f$ for some $m\in M$.  We claim that there exists
$K>0$ sufficiently large such that,
\begin{equation*}
\expL{m+f}=\expL{m+\sum_{\q|\alpha\w|\leq K} t^{\alpha}
f_{\alpha}\q( x\w) }.
\end{equation*}
Indeed, if $\q|\alpha\w|$ is so large that\footnote{Here, we are
thinking of $\q(\alpha,0\w)\in \N^N\times \N^n$.  Also, when we write $L\q(\expL{m+f} \w)$, we are dropping off the last coordinate of $\expL{m+f}$ so that the expression makes sense.}
$L\q(\alpha,0\w)>L\q(\expL{m+f}\w)$, then the term $t^{\alpha}
f_{\alpha}\q(x\w)$ does not affect $\expL{m+f}$.

Note, though, that $m+\sum_{\q|\alpha\w|\leq K} t^{\alpha}
f_{\alpha}\q( x\w)\in M$.  Thus, by definition,
$\expL{r}=\expL{m+f}\in E_L\q(M\w)$.  This achieves the
contradiction and completes the proof of
\eqref{EqnTaylorPrepFiniteSum}.

Now we turn to showing that the $c_{\alpha_k}$ may be modified so
that they satisfy \eqref{EqnTaylorPrepGoodCoef}.  Indeed, suppose
$c_{\alpha_k}$ satisfy \eqref{EqnTaylorPrepFiniteSum}.  Define
$\ch_{\alpha_k}$ by,
\begin{equation*}
t^{\alpha_k} \ch_{\alpha_k}\q( t,x\w) = t^{\alpha_k} c_{\alpha_k}\q(
t,x\w) -\sum_{j=1}^r \frac{t^{\alpha_j}}{\alpha_j!}
\q[\frac{\partial}{\partial s}^{\alpha_j}\bigg|_{s=0} s^{\alpha_k}
c_{\alpha_k}\q( s,x\w)\w] + t^{\alpha_k};
\end{equation*}
note that the right hand side is clearly of the form
$t^{\alpha_k}\ch_{\alpha_k}$ for some $\ch_{\alpha_k}$, since
$\frac{\partial}{\partial s}^{\alpha_j}\big|_{s=0} s^{\alpha_k}
\ch_{\alpha_k}\q( s,x\w)=0$ unless $\alpha_k\leq \alpha_j$
coordinatewise.

It is clear that $\ch_{\alpha_k}$ satisfies
\eqref{EqnTaylorPrepGoodCoef}.  Thus, to complete the proof, we need
only show that,
\begin{equation}\label{EqnTaylorPrepToShow}
\sum_{j=1}^r \ch_{\alpha_j}\q( t,x\w) t^{\alpha_j} f_{\alpha_j}\q(
x\w) = f\q( t,x\w).
\end{equation}
Since \eqref{EqnTaylorPrepToShow} holds with $\ch_{\alpha_j}$
replaced by $c_{\alpha_j}$, it suffices to show,
\begin{equation}\label{EqnTaylorPrepToShow2}
\sum_{k=1}^r t^{\alpha_k} f_{\alpha_k}\q( x\w) = \sum_{k=1}^r
\q(\sum_{j=1}^r \frac{t^{\alpha_j}}{\alpha_j!}
\q[\frac{\partial}{\partial s}^{\alpha_j}\bigg|_{s=0} s^{\alpha_k}
c_{\alpha_k}\q( s,x\w)\w]\w)f_{\alpha_k}\q( x\w).
\end{equation}
In light of \eqref{EqnTaylorPrepFiniteSum}, both sides of
\eqref{EqnTaylorPrepToShow2} are equal to,
\begin{equation*}
\sum_{j=1}^r \frac{t^{\alpha_j}}{\alpha_j!} \frac{\partial}{\partial
s}^{\alpha_j} \bigg|_{s=0} f\q( s,x\w).
\end{equation*}
This completes the proof.
\end{proof}

%% file: redcinf.tex
In this section, we use the results from Section
\ref{SectionRealAnal} to reduce Theorems \ref{ThmMainSingIntThm} and
\ref{ThmRealAnalMax} to 
theorems about $C^\infty$ $\gamma$:  namely, 
Theorems \ref{ThmSSLpBoundedness} and \ref{ThmMoreGenMaxFunc}.
In Section \ref{SectionReductionSingInt}, we reduce
Theorem \ref{ThmMainSingIntThm} to Theorem \ref{ThmSSLpBoundedness}
(or, more precisely, Proposition \ref{PropFiniteTypeAuto}),
while in Section \ref{SectionReductionMaxInt}
we reduce Theorem \ref{ThmRealAnalMax}
to Theorem \ref{ThmMoreGenMaxFunc}.

%% file: redsingint2.tex
In this section, we will show that Theorem \ref{ThmRealAnalMax}
follows from Theorem \ref{ThmSSLpBoundedness}.
In fact, as shown in Section \ref{SectionSSLp},
it suffices to prove Proposition \ref{PropFiniteTypeAuto}:  that
(\ref{ItemCondWithW}.F)
and (\ref{ItemCondFinal}.F) hold automatically when $\gamma$
is real analytic.  Furthermore, since Proposition \ref{PropCondWithWEquivCondFinal}
shows that (\ref{ItemCondWithW}.F) and (\ref{ItemCondFinal}.F) are equivalent,
it suffices to show that 
(\ref{ItemCondWithW}.F) holds whenever $\gamma$
is real analytic.

\begin{lemma}\label{LemmaRealAnalAlwaysGeneratesFiniteList}
Let $\sF=\q\{\q(X_1,d_1\w),\ldots, \q(X_r,d_r\w)\w\}$ be a finite set
of germs of real analytic vector fields (defined on a neighborhood of $0$),
each paired with formal degree $0\ne d_j\in \N^\nu$.  Then
$\sF$ generates a finite list, as in Definition \ref{DefnGeneratesAFiniteList}:
that is, there is a finite set of elements $\q(X_{r+1},d_{r+1}\w),\ldots,
\q(X_q,d_q\w)\in \sL_0\q(\sF\w)$ such that for every
$1\leq i,j\leq q$,
$\q(\q[X_i,X_j\w],d_i+d_j\w)$ is controlled by
$\q(X_1,d_1\w),\ldots, \q(X_q,d_q\w)$.
\end{lemma}
\begin{proof}
Apply Corollary \ref{CorNoeFiniteType} to 
$\sL\q(\sF\w)$ to obtain a finite set $\sF_0\subseteq \sL\q(\sF\w)$
such that $\sF_0$ controls every element of $\sL\q(\sF\w)$.
By Remark \ref{RmkJacobi}, we may assume $\sF_0\subseteq \sL_0\q(\sF\w)$.
We may, without loss of generality, replace $\sF_0$ with $\sF_0\cup \sF$.
We claim $\sF_0$ is the desired set $\q(X_1,d_1\w),\ldots, \q(X_q,d_q\w)$.
Indeed, it only remains to show that
for every $\q(Y_1,f_1\w),\q(Y_2,f_2\w)\in \sF_0$, $\q(\q[Y_1,Y_2\w],f_1+f_2\w)$
is controlled by $\sF_0$.  Since $\q(\q[Y_1,Y_2\w],f_1+f_2\w)\in \sL\q(\sF\w)$, this follows
by the definition of $\sF_0$, completing the proof.
\end{proof}

\begin{proof}[Proof of Proposition \ref{PropFiniteTypeAuto}]
We take $W$ as in Section \ref{SectionSSLp}.  Since we are assuming
$\gamma$ is real analytic, $W$ is real analytic.
We express $W$ as a Taylor series in the $t$ variable:
\begin{equation*}
W\q(t\w) = \sum_{\q|\alpha\w|>0} t^{\alpha}\Xh_\alpha,
\end{equation*}
where the $\Xh_\alpha$ are real analytic vector fields.
The goal is to show that there is a finite set,
\begin{equation}\label{EqnToShowFiniteSet}
\sF\subseteq \q\{\q(\Xh_\alpha,\deg\q(\alpha\w)\w)\w\},
\end{equation}
such that $\sF$ generates a finite list, and this finite list controls $W$.
Since the vector fields are real analytic, Lemma \ref{LemmaRealAnalAlwaysGeneratesFiniteList}
shows that $\sF$ automatically generates a finite list.  Thus, it suffices
to show that there is a finite set $\sF$ as in \eqref{EqnToShowFiniteSet}
such that $\sF$ controls $W$.
We apply Theorem \ref{ThmTaylorPrep} to $W$ to show that there
exist $\alpha_1,\ldots, \alpha_r$ such that,
\begin{equation*}
W\q(t,x\w) = \sum_{j=1}^r c_j\q(t,x\w) t^{\alpha_j} X_{\alpha_j}\q(x\w).
\end{equation*}
From here, it is immediate to verify
that $\q(X_{\alpha_1},\deg\q(\alpha_1\w)\w),\ldots,\q(X_{\alpha_r},\deg\q(\alpha_r\w)\w)$ control $W$, completing the proof.
\end{proof}

We close this section by proving Proposition
\ref{PropClosedUnderComp}.  The main point is the following. Suppose
we are given $\nu_1$ parameter dilations on $\R^{N_1}$ and $\nu_2$
parameter dilations on $\R^{N_2}$, and suppose we are given
$\gamma_{t_j}^j\q(x\w):\R^{N_j}_0\times\R^n_0\rightarrow \R^n$,
$j=1,2$, germs of real analytic functions, satisfying the hypotheses
of Theorem \ref{ThmMainSingIntThm} (i.e,
satisfying (\ref{ItemCondFinal}.A)).  Proposition
\ref{PropClosedUnderComp} will follow if we show that
$\gamma_{t_1}^1\circ \gamma_{t_2}^2\q(x\w)$ and
$\q(\gamma_{t_1}^1\w)^{-1}\q(x\w)$ both satisfy the 
(\ref{ItemCondFinal}.A).
For $\gamma_{t_1}^1\circ
\gamma_{t_2}^2$ we are using the $\nu_1+\nu_2$ parameter dilations
on $\R^{N_1+N_2}$ given by, for $\q(\delta_1,\delta_2\w)\in
\q[0,1\w]^{\nu_1}\times \q[0,1\w]^{\nu_2}$ and $\q(t_1,t_2\w)\in
\R^{N_1}\times\R^{N_2}$,
$\q(\delta_1,\delta_2\w)\q(t_1,t_2\w)=\q(\delta_1 t_1,\delta_2
t_2\w)$. We write,
\begin{equation*}
\gamma_{t_j}^{j}\q(x\w)\sim \exp\q(\sum_{\q|\alpha\w|>0}
t_j^{\alpha} X_\alpha^j\w)x.
\end{equation*}
Lemma 9.3 of
\cite{ChristNagelSteinWaingerSingularAndMaximalRadonTransforms}
shows that,
\begin{equation*}
\q(\gamma_{t_1}^1\w)^{-1}\q(x\w) \sim \exp\q(\sum_{\q|\alpha\w|>0}
-t_j^{\alpha} X_\alpha^j\w)x.
\end{equation*}
The fact that $\q(\gamma_{t_1}^1\w)^{-1}$ satisfies 
(\ref{ItemCondFinal}.A)
now follows immediately.

We now turn to $\gamma_{t_1}^1\circ \gamma_{t_2}^2$. 
Define
\begin{equation*}
\sT=\sL\q( \q\{ \q(X_{\alpha_1}^1 ,\q(\alpha_1,0\w)\w),\q(X_{\alpha_2}^2,\q(0,\alpha_2\w)\w)\w\} \w).
\end{equation*}
 It follows from the Campbell-Hausdorff
formula (see Section 3 of
\cite{ChristNagelSteinWaingerSingularAndMaximalRadonTransforms})
that,
\begin{equation}\label{EqnG1circG1Asm}
\gamma_{t_1}^1\circ \gamma_{t_2}^2\q( x\w) \sim
\exp\q(\q[\sum_{\q|\alpha_1\w|>0} t_1^{\alpha_1}
X_{\alpha_1}^1\w]+\q[\sum_{\q|\alpha_2\w|>0} t_2^{\alpha_2}
X_{\alpha_2}^2\w] + \q[\sum_{\q|\beta_1\w|,\q|\beta_2\w|>0}
t_1^{\beta_1} t_2^{\beta_2} X_{\beta_1,\beta_2}\w] \w)x,
\end{equation}
where
$X_{\beta_1,\beta_2}\in \Span{ X' : \q(X',\q(\beta_1,\beta_2\w)\w)\in \sT }$. 

Define,
\begin{equation*}
\begin{split}
\sP_1&=\q\{\q(X_{\alpha_1}^{1},\q(\deg\q(\alpha_1\w),0\w)\w):
\deg\q(\alpha_1\w) \text{ is nonzero in only one component}\w\}\subseteq \sA_n^n\times \N^{\nu_1+\nu_2},\\
\sP_2&=\q\{\q(X_{\alpha_2}^{2},\q(0,\deg\q(\alpha_2\w)\w)\w):
\deg\q(\alpha_2\w) \text{ is nonzero in only one
component}\w\}\subseteq \sA_n^n\times \N^{\nu_1+\nu_2};
\end{split}
\end{equation*}
where $\deg\q(\alpha_1\w)$ is defined with the dilations on
$\R^{N_1}$ and $\deg\q(\alpha_2\w)$ is defined with the dilations on
$\R^{N_2}$. In light of \eqref{EqnG1circG1Asm} the vector fields
associated to the pure powers of $\gamma_{t_1}^1\circ
\gamma_{t_2}^2$ are given by $\sP_1\cup \sP_2$. 
Our assumption that $\gamma^1$ and $\gamma^2$ satisfy 
(\ref{ItemCondFinal}.A)
imply that $\q( X_{\alpha_1}^1,
\q(\deg\q(\alpha_1\w),0\w)\w)$ and $\q( X_{\alpha_2}^2,
\q(0,\deg\q(\alpha_2\w)\w)\w)$ are controlled by $\sL\q(\sP_1\cup \sP_2\w)$ 
for every $\alpha_1$ and $\alpha_2$.  Since
every element of $\sT$ is given by iterated commutators of
$X_{\alpha_1}^1$ and $X_{\alpha_2}^2$, it follows that for every
$\q( X,\q(\beta_1,\beta_2\w)\w)\in \sT$, $\sL\q(\sP_1\cup \sP_2\w)$ controls $\q(
X,\q(\deg\q(\beta_1\w),\deg\q(\beta_2\w)\w)\w)$.
Hence, $\sL\q(\sP_1\cup \sP_2\w)$ controls $\q( X_{\beta_1,\beta_2},
\q(\deg\q(\beta_1\w),\deg\q(\beta_2\w)\w)\w)$, 
for every $\beta_1$ and $\beta_2$.  Thus, $\gamma_{t_1}^1\circ
\gamma_{t_2}^2$ satisfies 
(\ref{ItemCondFinal}.A).
This completes the proof of Proposition
\ref{PropClosedUnderComp}.

%% file: redmax.tex
In this section, we reduce Theorem \ref{ThmRealAnalMax} to Theorem
\ref{ThmMoreGenMaxFunc}.  The main tool will be Theorem
\ref{ThmTaylorPrep}.

Let $\gamma:\R_0^N\times \R_0^n\rightarrow \R^n$ be a germ of a real
analytic function satisfying $\gamma_0\q(x\w)\equiv x$.  For
$\delta=\q(\delta_1,\ldots, \delta_N\w)\in \q[0,1\w]^N$ and $t=\q(
t_1,\ldots, t_N\w)\in \R^N$, we define $\delta t=\q( \delta_1
t_1,\ldots, \delta_N t_N\w)$.  The goal is to study the maximal
operator,
\begin{equation*}
\sM f\q(x\w) = \sup_{\delta\in
\q[0,1\w]^N}\psi_1\q(x\w)\int_{\q|t\w|\leq a} \q|f\q( \gamma_{\delta
t}\q(x\w)\w)\w|\: dt.
\end{equation*}
Where $\psi_1\in C_0^\infty\q( \R^n\w)$ is supported on a small
neighborhood of $0$ and $\psi_1\geq 0$.  Let $\psi_2\in
C_0^\infty\q( \R^n\w)$, $\psi_2\geq 0$, with $\psi_2\equiv 1$ on a
neighborhood of the support of $\psi_1$.  We may assume $\psi_2$ has
small support, by shrinking the support of $\psi_1$.  By taking
$a>0$ small, we may ensure for $\q|t\w|<a$, and $x$ in the support
of $\psi_1$, we have $\psi_2\q( \gamma_t\q(x\w)\w)=1$.  With this
setup, define,
\begin{equation*}
\sM_0 f\q(x\w) = \sup_{j\in \N^{N}}\psi_1\q(x\w)\int_{\q|t\w|\leq a}
\q|f\q( \gamma_{2^{-j}t}\q(x\w)\w) \w| \psi_2\q(
\gamma_{2^{-j}t}\q(x\w)\w) \: dt.
\end{equation*}
It is easy to see that we have the pointwise inequality, $\sM
f\q(x\w) \lesssim \sM_0 f\q(x\w)$.  Thus, to prove Theorem
\ref{ThmRealAnalMax}, it suffices to prove $\sM_0$ is bounded on
$L^p$, $1<p\leq \infty$.

Define the real analytic vector field,
\begin{equation*}
W\q( t,x\w) = \frac{d}{d\epsilon}\bigg|_{\epsilon=1}
\gamma_{\epsilon t}\circ \gamma_t^{-1}\q(x\w)\in T_x\R^n.
\end{equation*}
Note that, for $j\in \N^N$,
\begin{equation}\label{EqnWDilated}
W\q( 2^{-j}t,x\w) = \frac{d}{d\epsilon}\bigg|_{\epsilon=1}
\gamma_{\epsilon 2^{-j}t}\circ \gamma_{2^{-j}t}^{-1}\q(x\w)\in
T_x\R^n.
\end{equation}
That is, replacing $\gamma_t$ with $\gamma_{2^{-j}t}$ changes
$W\q(t,x\w)$ to $W\q( 2^{-j} t,x\w)$.

Write,
\begin{equation*}
W\q(t,x\w)=\sum_{\q|\alpha\w|>0} t^{\alpha} X_\alpha.
\end{equation*}
 Applying Theorem
\ref{ThmTaylorPrep} to $W\q(t,x\w)$ we see that there exist
$\alpha_1,\ldots, \alpha_r$ and germs of real analytic functions
$c_{\alpha_l}$ such that,
\begin{equation}\label{EqnWDecomposed}
W\q(t,x\w) = \sum_{l=1}^r c_{\alpha_l}\q(t,x\w) t^{\alpha_l}
X_{\alpha_l}.
\end{equation}
Moreover, we may assume that the $c_{\alpha_l}$ satisfy
\eqref{EqnTaylorPrepGoodCoef}.

Let $\nu=N+r$.  We will define $\nu$-parameter dilations on $W$. For
$l=1,\ldots, r$, let $\hd_l\in \N^{r}$ be equal to $1$ in the $l$
component and $0$ in all other components.  Then, for
$\q(j_1,j_2\w)\in \Ninf^N\times \Ninf^{r}$, define,
\begin{equation*}
W_{\q(j_1,j_2\w)}\q(t,x\w)= \sum_{l=1}^r
c_{\alpha_l}\q(2^{-j_1}t,x\w) t^{\alpha_l} 2^{-j_2\cdot \hd_l}
X_{\alpha_l}.
\end{equation*}
Let $\gamma^{\q(j_1,j_2\w)}_t$ be the function corresponding to
$W_{\q(j_1,j_2\w)}$ as in Proposition \ref{PropBijectGammaW}.  Just
as in Section \ref{SectionMoreGenMax}, it is easy to see (via the
contraction mapping principle) that there exist open sets $0\in
U\subseteq \R^N$, $0\in V\subseteq \R^n$, independent of $j_1,j_2$
such that $\gamma^{j_1,j_2}:U\times V\rightarrow \R^n$.  By possibly
shrinking $a$ and the support of $\psi_1,\psi_2$, we may define the
maximal function,
\begin{equation*}
\sM_1 f\q(x\w)= \sup_{\q(j_1,j_2\w)\in \N^{N}\times
\N^r}\psi_1\q(x\w)\int_{\q|t\w|\leq a} \q|f\q(
\gamma_{t}^{\q(j_1,j_2\w)}\q(x\w)\w) \w| \psi_2\q(
\gamma^{\q(j_1,j_2\w)}_{t}\q(x\w)\w) \: dt.
\end{equation*}

We claim that $\sM_0 f\q(x\w)\leq \sM_1 f\q(x\w)$. To see this, we need
only show for every $j\in \N^{N}$, $\gamma_{2^{-j}t}$ is of the form
$\gamma^{\q(j_1,j_2\w)}_t$ for some $j_1,j_2$. In light of
\eqref{EqnWDilated}, it suffices to show for every $j\in \N^{N}$,
$W\q(2^{-j}t,x\w)$ is of the form $W_{\q(j_1,j_2\w)}\q(t,x\w)$ for
some $j_1,j_2$. In light of \eqref{EqnWDecomposed},
\begin{equation*}
W\q(2^{-j}t,x\w)= \sum_{l=1}^r c_{\alpha_l}\q(2^{-j}t,x\w)
t^{\alpha_l} 2^{-j\cdot \alpha_l} X_{\alpha_l}.
\end{equation*}
Thus, if we take $j_2=\q( j\cdot \alpha_1,j\cdot\alpha_2,\ldots,
j\cdot\alpha_r\w)$, we have $W\q(2^{-j}t,x\w) =
W_{\q(j,j_2\w)}\q(t,x\w)$. This completes the proof that $\sM_0
f\q(x\w)\leq \sM_1 f\q(x\w)$.

Hence, to prove Theorem \ref{ThmRealAnalMax}, we need only show that
$\sM_1$ is bounded on $L^p$, $1<p\leq \infty$.  We will show that
$\sM_1$ is of the form covered in Theorem \ref{ThmMoreGenMaxFunc},
thereby reducing Theorem \ref{ThmRealAnalMax} to Theorem
\ref{ThmMoreGenMaxFunc}.

For $l=1,\ldots, r$, let $X_l=X_{\alpha_l}$, $c_l\q(t,s,x\w)=
c_{\alpha_l}\q( t,x\w) s^{\alpha_l}$, and $d_l\in
\N^{\nu}=\N^{N}\times \N^r$ be given by
$d_l=\q(0,\hd_l\w)\in\N^N\times \N^r$.  Furthermore, for
$\q(j_1,j_2\w)\in \N^{N}\times \N^r$, we define $2^{-\q(j_1,j_2\w)}t
= 2^{-j_1}t$. With this new notation, for $j\in \N^\nu$, we have
\begin{equation*}
W_j\q(t,x\w)= \sum_{l=1}^r c_l\q(2^{-j} t,t,x\w) 2^{-j\cdot d_l}X_l.
\end{equation*}

We apply Lemma \ref{LemmaRealAnalAlwaysGeneratesFiniteList}
to extend the list $\q(X_1,d_1\w), \ldots, \q(X_r,d_r\w)$
to a list $\q(X_1,d_1\w),\ldots, \q(X_q,d_q\w)$
as in Lemma \ref{LemmaRealAnalAlwaysGeneratesFiniteList}:  that this
extended list satisfies the hypotheses of the list of the
same name in Section \ref{SectionMoreGenMax} is exactly
the conclusion of Lemma \ref{LemmaRealAnalAlwaysGeneratesFiniteList}.

For $r+1\leq l\leq q$, define $c_l\q(t,s,x\w)\equiv 0$.  Note, we
have,
\begin{equation*}
W_j\q( t,x\w) = \sum_{l=1}^q c_l\q( 2^{-j} t,t,x\w) 2^{-j\cdot d_l}
X_l.
\end{equation*}
To complete the proof that $\sM_1$ satisfies the hypotheses of
Theorem \ref{ThmMoreGenMaxFunc}, we need only show that $c_1,\ldots,
c_r$ satisfy \eqref{EqnGenMaxCoefNonVanish} and $c_1,\ldots, c_q$
satisfy \eqref{EqnGenMaxCoefVanish}.  \eqref{EqnGenMaxCoefVanish} is
trivial for $c_{r+1},\ldots, c_q$ (since they are all $0$) and so we
need only verify \eqref{EqnGenMaxCoefNonVanish} and
\eqref{EqnGenMaxCoefVanish} for $c_1,\ldots, c_r$.  Here, we are
taking $\alpha_1,\ldots, \alpha_r$ as above (see
\eqref{EqnWDecomposed}).  \eqref{EqnGenMaxCoefNonVanish} and
\eqref{EqnGenMaxCoefVanish} will follow from the fact that
$c_{\alpha_1},\ldots, c_{\alpha_r}$ satisfy
\eqref{EqnTaylorPrepGoodCoef}.

First we verify \eqref{EqnGenMaxCoefNonVanish}.  Let $1\leq l\leq
r$.  Consider,
\begin{equation*}
\frac{1}{\alpha_l!} \frac{\partial}{\partial
s}^{\alpha_l}\bigg|_{s=t=0}c_l\q(t,s,x\w) = \frac{1}{\alpha_l!}
\frac{\partial}{\partial
s}^{\alpha_l}\bigg|_{s=t=0}c_{\alpha_l}\q(t,x\w) s^{\alpha_l}=
\frac{1}{\alpha_l!} \frac{\partial}{\partial
t}^{\alpha_l}\bigg|_{t=0}c_{\alpha_l}\q(t,x\w) t^{\alpha_l}=1,
\end{equation*}
where the last equality follows from \eqref{EqnTaylorPrepGoodCoef}.
Thus, \eqref{EqnGenMaxCoefNonVanish} holds.

We turn to \eqref{EqnGenMaxCoefVanish}.  Fix $1\leq l,k\leq r$ and
$\beta_1,\beta_2$ such that $\beta_1+\beta_2=\alpha_l$.  Consider,
\begin{equation}\label{EqnShowingCoefVanish1}
\frac{\partial}{\partial t}^{\beta_1} \frac{\partial}{\partial
s}^{\beta_2}\bigg|_{s=t=0} c_{k}\q( t,s,x\w) =
\frac{\partial}{\partial t}^{\beta_1} \frac{\partial}{\partial
s}^{\beta_2}\bigg|_{s=t=0} c_{\alpha_k}\q( t,x\w) s^{\alpha_k}.
\end{equation}
Note that the right hand side of \eqref{EqnShowingCoefVanish1} is
$0$ unless $\beta_2=\alpha_k$.  Thus, we need only consider the case
when $\beta_2=\alpha_k$; in this case, we have,
\begin{equation}\label{EqnShowingCoefVanish2}
\frac{\partial}{\partial t}^{\beta_1} \frac{\partial}{\partial
s}^{\beta_2}\bigg|_{s=t=0} c_{\alpha_k}\q( t,x\w) s^{\alpha_k} = C
\frac{\partial}{\partial t}^{\alpha_l} \bigg|_{t=0} c_{\alpha_k}\q(
t,x\w) t^{\alpha_k},
\end{equation}
where $C$ is some constant.  Note that the right hand side of
\eqref{EqnShowingCoefVanish2} is $0$ unless $l=k$, by
\eqref{EqnTaylorPrepGoodCoef}.  \eqref{EqnGenMaxCoefVanish} follows.
This completes the proof that $\sM_1$ is of the form covered by
Theorem \ref{ThmMoreGenMaxFunc}, and finishes the reduction of
Theorem \ref{ThmRealAnalMax} to Theorem \ref{ThmMoreGenMaxFunc}.

\begin{rmk}
Let us take a moment to remark on the essential idea of this section.
When one is considering the singular Radon transform (Theorem \ref{ThmMainSingIntThm}), which vector fields correspond to pure powers and non-pure powers
is forced, due to the nature of the cancellation in the singular kernel.
However, when we consider the maximal function, we introduce
the cancellation in an {\it ad hoc} way (see the operators $B_j$
in Section \ref{SectionProofOfMax}).  Because of this, we have
some freedom in choosing which vector fields correspond
to pure powers, by considering a stronger maximal operator.
This idea was adapted from \cite{ChristTheStrongMaximalFunctionOnANilpotentGroup}.
\end{rmk}

%% file: proofofmax.tex
In this section, we prove Theorem \ref{ThmMoreGenMaxFunc}.  The
proof is a modification of the proof of Theorem 5.4 of
\cite{SteinStreetMultiParameterSingRadonLp}.  First we will
introduce some necessary auxiliary operators, in a manner completely
analogous to the methods in
\cite{SteinStreetMultiParameterSingRadonLp}.  Then, we will describe
the modifications of the proof in
\cite{SteinStreetMultiParameterSingRadonLp} necessary to prove
Theorem \ref{ThmMoreGenMaxFunc}.  The reader may wish to have a copy
of \cite{SteinStreetMultiParameterSingRadonLp} at hand, as we will
be referring to it repeatedly.

The proof of Theorem \ref{ThmMoreGenMaxFunc} proceeds by induction
on $\nu$.  We begin by describing the necessary modifications to
Section 9 of \cite{SteinStreetMultiParameterSingRadonLp}, where the
induction is set up. We take all the same notation as Theorem
\ref{ThmMoreGenMaxFunc}. Let $\denum{\psi}{0}\in
C_0^\infty\q(\R^n\w)$ be non-negative and satisfy
$\psi_1,\psi_2\prec\denum{\psi}{0}$.  We also assume that
$\denum{\psi}{0}$ has small support. Let $\sigma\in C_0^\infty\q(
B^N\q(a\w)\w)$ satisfy $\sigma\geq 0$ and $\sigma\geq 1$ on a
neighborhood of $0$.
We define
for $j\in \Ninf^\nu$,
\begin{equation*}
M_j f\q(x\w) = \denum{\psi}{0}\q(x\w)\int f\q(\gamma^j_t\q(x\w)\w)
\denum{\psi}{0}\q(\gamma^j_t\q(x\w)\w)\sigma\q(t\w)\: dt.
\end{equation*}
It is immediate to see, if we shrink $a>0$ in the definition of
$\sMt$, we have,
\begin{equation*}
\sMt f\q(x\w) \lesssim \sup_{j\in \N} M_j \q|f\w| \q(x\w).
\end{equation*}
Thus, to prove Theorem \ref{ThmMoreGenMaxFunc} it suffices to prove
the following proposition,
\begin{prop}\label{PropMaxAsVect}
\begin{equation*}
\LpN{p}{\sup_{j\in \N^\nu}\q|M_j f\q(x\w)\w|}\lesssim \LpN{p}{f},
\end{equation*}
for $1<p<\infty$.
\end{prop}
Indeed, merely apply Proposition \ref{PropMaxAsVect} to $\q|f\w|$ to
prove Theorem \ref{ThmMoreGenMaxFunc}. It is Proposition
\ref{PropMaxAsVect} which we prove by induction on $\nu$.  For
$E\subseteq \nuset$ and $j=\q(j^1,\ldots, j^\nu\w)\in \N^\nu$,
define $j_E=\q(j_E^1,\ldots, j_E^\nu\w)\in \Ninf^\nu$ by,
\begin{equation*}
j_E^\mu = \begin{cases} j^\mu &\text{if }\mu\in E,\\
\infty & \text{if }\mu\not\in E.\end{cases}
\end{equation*}

Thinking of $j_E$ as an element of $\N^{\q|E\w|}$ (by suppressing
those coordinates which equal $\infty$), it is easy to see that
$M_{j_E}$ is of the same form as $M_j$, but with $\nu$ replaced by
$E$.  In particular, $\gamma_t^{j_E}$ is of the same form as
$\gamma_t^j$, but instead with $\q|E\w|$ parameter dilations. Thus
our inductive hypothesis implies for $E\subsetneq \nuset$,
\begin{equation*}
\LpN{p}{\sup_{j\in \N^\nu}\q|M_{j_E} f\q(x\w)\w|}\lesssim
\LpN{p}{f},\quad 1<p<\infty.
\end{equation*}
Note that the base case of our induction is trivial.  Indeed,
$M_{j_\emptyset} f = \q[\int \sigma\q(t\w)\: dt\w] \denum{\psi}{0}^2
f$.

For $j\in \Ninf^\nu$, we define $A_j$ from the list of vector fields
$\q(X,d\w)$ just as in \cite{SteinStreetMultiParameterSingRadonLp}.
Similarly, for $j\in \N^\nu$, we define $D_j$ just as in
\cite{SteinStreetMultiParameterSingRadonLp}.  For $j\in \N^{\nu}$,
we define,
\begin{equation*}
B_j=\sum_{E\subseteq \nuset} \q(-1\w)^{\q|E\w|} A_{j_{E^c}} M_{j_E}.
\end{equation*}
Just as in \cite{SteinStreetMultiParameterSingRadonLp}, Proposition
\ref{PropMaxAsVect} follows from the following proposition,
\begin{prop}\label{PropToShowB}
\begin{equation*}
\LpN{p}{\sup_{j\in \N^\nu}\q|B_j f\q(x\w)\w|}\lesssim \LpN{p}{f},
\end{equation*}
for $1<p<\infty$.
\end{prop}

It follows in exactly the same manner as
\cite{SteinStreetMultiParameterSingRadonLp} that to prove
Proposition \ref{PropToShowB}, it suffices to prove,
\begin{prop}\label{PropL2Cancel}
If $a>0$ is sufficiently small, there exists $\epsilon>0$ such that,
\begin{equation*}
\LpOpN{2}{B_{j} D_k}\lesssim 2^{-\epsilon\q|j-k\w|},
\end{equation*}
for $j,k\in \N^{\nu}$.
\end{prop}

The proof 
in \cite{SteinStreetMultiParameterSingRadonLp}
of the result analogous to Proposition \ref{PropL2Cancel}
(Theorem 10.1 of
\cite{SteinStreetMultiParameterSingRadonLp}) follows by reducing the
question to a general result in
\cite{StreetMultiParameterSingRadonLt}.  The proof of Proposition
\ref{PropL2Cancel} has the same basic outline as the proof of
Theorem 10.1 of \cite{SteinStreetMultiParameterSingRadonLp}, with
only a few minor differences.  We outline, below, the necessary
facts needed to adapt the proof in
\cite{SteinStreetMultiParameterSingRadonLp} to our situation.  Note
that $A_j$ and $D_j$ are defined in the same way as in
\cite{SteinStreetMultiParameterSingRadonLp}--we therefore only need
to discuss the modifications necessary to deal with the new form of
$M_j$.

One key point is the following; for $1\leq l\leq r$, and $j'\in
\Ninf^\nu$,
\begin{equation}\label{EqnXjAsTaylor}
\frac{1}{\alpha_l!}\frac{\partial}{\partial
t}^{\alpha_l}\bigg|_{t=0} W_{j'}\q(t,x\w) = 2^{-{j'}\cdot d_l} X_l;
\end{equation}
i.e., $2^{-{j'}\cdot d_l} X_l$ is the Taylor coefficient of
$t^{\alpha_l}$, when the Taylor series of $W_{j'}$ is taken in the
$t$ variable.  This is an immediate consequence of the definition of
$W_j$ and \eqref{EqnGenMaxCoefNonVanish} and
\eqref{EqnGenMaxCoefVanish}.  \eqref{EqnXjAsTaylor} is the main
property needed when $M_{j_E}$ plays the role of some $S_l$ in
\cite{SteinStreetMultiParameterSingRadonLp}.

\begin{rmk}
One also needs that ``$M_{j_E}$ is controlled by $\q( 2^{-j\wedge k}
X,\sd\w)$ at the unit scale'' (see
\cite{SteinStreetMultiParameterSingRadonLp} for this terminology).
This follows immediately from the definition of $M_{j_E}$.
\end{rmk}

The other main property we need is as follows.  In the case when
$j^{\mu_1}-k^{\mu_1}=\q|j-k\w|_\infty$, for some $\mu_1$, we must
use $M_{j_{\Ewmu}}-M_{j_E}$ as $R_1-R_2$ in the argument (see
\cite{SteinStreetMultiParameterSingRadonLp} for a discussion of what
we mean by $R_1-R_2$).
Define, $\Wh_{j,k,E,\mu_1}\q( t,s,x\w)$ by the same formula as
$W_{j_E}\q(t,x\w)$ except with $2^{-j_E}$ replaced by
$\delta=\q(\delta_1,\ldots,\delta_\nu\w)\in \q[0,1\w]^\nu$, where,
\begin{equation*}
\delta_{\mu}=\begin{cases} 2^{-j_E^\mu} & \text{if }\mu\ne \mu_1,\\
s2^{-k^{\mu_1}} & \text{if }\mu=\mu_1.
\end{cases}
\end{equation*}
Note that,
\begin{equation*}
\Wh_{j,k,E,\mu_1}\q( t,0,x\w)=W_{j_E}\q(t,x\w),
\quad\Wh_{j,k,E,\mu_1}\q( t,2^{k^{\mu_1}-j^{\mu_1}},x\w)=
W_{j_{\Ewmu}}\q(t,x\w).
\end{equation*}
Thus, if we let $\gh_{t,s}\q(x\w)$ be the function associated to
$\Wh_{j,k,E,\mu_1}\q( t,s,x\w)$ as in Proposition
\ref{PropBijectGammaW}, we see that,
\begin{equation*}
M_{j_E}f\q(x\w) = \denum{\psi}{0}\q(x\w) \int
f\q(\gh_{t,0}\q(x\w)\w)\denum{\psi}{0}\q(\gh_{t,0}\q(x\w)\w)\sigma\q(t\w)\:
dt,
\end{equation*}
\begin{equation*}
M_{j_{\Ewmu}}f\q(x\w) = \denum{\psi}{0}\q(x\w) \int
f\q(\gh_{t,2^{k^{\mu_1}-j^{\mu_1}}}\q(x\w)\w)\denum{\psi}{0}\q(\gh_{t,2^{k^{\mu_1}-j^{\mu_1}}}\q(x\w)\w)\sigma\q(t\w)\:
dt.
\end{equation*}
From here it is easy to see that $M_{j_{\Ewmu}}-M_{j_E}$ can play
the role of $R_1-R_2$ in this situation.

With the above outlined modifications, the proof in
\cite{SteinStreetMultiParameterSingRadonLp} goes through to prove
Proposition \ref{PropL2Cancel}.  We leave the details to the
interested reader.  This completes the proof of Proposition
\ref{PropL2Cancel} and therefore the proof of Theorem
\ref{ThmMoreGenMaxFunc}.

%% file: closing.tex
In this paper, we put more restrictions on the classes of kernels
$K$ we considered, as compared to \cite{SteinStreetMultiParameterSingRadonLp}.
None of these additional restrictions were essential.

In \cite{SteinStreetMultiParameterSingRadonLp}, the class of
kernels $\sK\q(N,e,a,\nu\w)$ was allowed to depend on another
parameter $\mu_0$ ($1\leq \mu_0\leq \nu$).  In this paper,
we have restricted to the case $\mu_0=\nu$.
All of the methods in this paper transfer seamlessly
over to the case of general $\mu_0$; we leave such details
to the interested reader.

In \cite{SteinStreetMultiParameterSingRadonLp}, the coordinates
of the dilations $e_j$ were allowed to be elements of $\q[0,\infty\w)$,
instead of $\N$.  This assumption was used in the proof
of Theorem \ref{ThmFiniteGenControl}, but nowhere else.
To deal with more general $e_j$, Theorem \ref{ThmFiniteGenControl}
can be replaced by the following proposition.

\begin{prop}
Suppose
\begin{equation*}
\sS \subseteq \sA_N^n\times \q[0,\infty\w)^\nu
\end{equation*}
is such that for every $M$, the set
\begin{equation*}
\sC_M:=\q\{c\in \q[0,M\w] : \exists \q(Y,d_0\w)\in \sS\text{ with some coordinate of }d_0 \text{ equal to }c\w\}
\end{equation*}
is finite.
Then there exists a finite subset $\sF\subseteq \sS$ such that every
$\q(g,e\w)\in \sS$
can be written in the form,
\begin{equation}\label{EqnToShowSpanClosing}
g\q(x\w)=\sum_{\substack{\q(f,d\w)\in \sF\\d\leq e}} c_{\q(f,d\w)}\q(x\w) f\q(x\w);
\end{equation}
where $c_{\q(f,d\w)}\in \sA_N$, and $d\leq e$ means that the inequality
holds for each coordinate.
The neighborhood on which \eqref{EqnToShowSpanClosing} holds
may depend on $\q(g,e\w)$.
\end{prop}
\begin{proof}
The proof proceeds by induction on $\nu$.  The base case, $\nu=0$,
follows directly from Proposition \ref{PropNoeMod}.
We assume we have the result for $\nu-1$ and prove it for $\nu$.

Let $\sM$ be the module generated by $\q\{Y:\exists \q(Y,d\w)\in \sS\w\}$.
By Proposition \ref{PropNoeMod}, $\sM$ is finitely generated.
Take $\q(X_1,d_1\w),\ldots, \q(X_r,d_r\w)\in \sS$ such that
$X_1,\ldots, X_r$ generate $\sM$.
Define $M=\max_{1\leq l \leq r} \q|d_l\w|_\infty$,
and let $c_1,\ldots, c_L$ be an enumeration of $\sC_M$.
Define,
\begin{equation*}
\sS_0:=\q\{\q(Y,d\w)\in \sS : \text{every coordinate of d is }> M\w\},
\end{equation*}
and for $1\leq \mu\leq \nu$, $1\leq l\leq L$,
\begin{equation*}
\sS_\mu^l := \q\{\q(Y,d\w)\in \sS : \text{the }\mu\text{ coordinate of }d\text { equals }c_l\w\}.
\end{equation*}
By our assumption on $\sS$,
\begin{equation}\label{EqnSPartitioned}
\sS=\sS_0\bigcup \q[\bigcup_{\mu=1}^\nu \bigcup_{l=1}^L \sS_\mu^l\w].
\end{equation}
Note that every $\q(Y,e\w)\in \sS_0$ can be written in the form,
\begin{equation}\label{EqnS0part}
Y=\sum_{d_j\leq e} c_j X_j,
\end{equation}
by our construction of $\sS_0$.

We apply our inductive hypothesis to $\sS_\mu^l$ (which we may think of
as a subset of $\sA_N^n\times\q[0,\infty\w)^{\nu-1}$ by suppressing
the $\mu$th coordinate of $d$ for each $\q(Y,d\w)\in \sS_\mu^l$,
since we know it to be equal to $c_l$).
We therefore obtain a finite subset $\sF_{\mu}^l\subseteq\sS_\mu^l$,
as in the conclusion of the proposition (with $\sS$ replaced by $\sS_\mu^l$).

By \eqref{EqnS0part} and \eqref{EqnSPartitioned} it is immediate
to verify that
\begin{equation*}
\q\{\q(X_1,d_1\w),\ldots, \q(X_l,d_l\w)\w\}\bigcup\q[ \bigcup_{\mu=1}^\nu \bigcup_{l=1}^L \sF_\mu^l\w]
\end{equation*}
satisfies the conclusion of the proposition.
\end{proof}